\documentclass[a4, times, 12p]{amsart}
\usepackage[english]{babel}
\usepackage{a4wide}
\usepackage{txfonts}
\usepackage{amsmath}
\usepackage{amssymb}
\usepackage{amsthm}
\usepackage{color}
\usepackage[title]{appendix}

\def\nc{\newcommand}

\newcounter{licznik}[section]

\newtheorem{thm}[licznik]{THEOREM}
\newtheorem{prop}[licznik]{PROPOSITION}
\newtheorem{lemma}[licznik]{LEMMA}
\newtheorem{cor}[licznik]{COROLLARY}
\theoremstyle{definition}

\theoremstyle{remark}
\newtheorem{remark}[licznik]{REMARK}
\newtheorem{example}[licznik]{EXAMPLE}

\makeatletter
\newcommand\assumptionlabel[1]{\hspace\labelsep
                               \normalfont\bfseries #1 \gdef\@currentlabel{#1}}
\newenvironment{assumption}
               {\list{}{\labelwidth\z@ \itemindent-\leftmargin
                        }}
               {\endlist}
\makeatother

\nc\com[1]{\textcolor{red}{#1}}
\nc\mop[1]{\hspace{#1pt}}
\nc{\laa}{\mathcal A}
\nc{\gie}{\mathcal G}
\nc{\lo}{\mathcal O}
\nc\ha{\mathcal{H}}
\nc\gk{\mathcal G}
\nc\ga{\mathcal A}
\nc{\prob}{\mathbb{P}}
\nc\pe{\mathcal{P}}
\nc{\lde}{\mathcal{D}}
\nc\eh{\mathcal{H}}
\nc{\ef}{\mathcal{F}}
\nc{\er}{\mathbb{R}}
\nc{\erp}{\mathbb{R}_+}
\nc{\en}{{\mathbb{N}}}
\nc{\es}{\mathcal{S}}
\nc{\te}{\mathcal{T}}
\nc{\borel}{\mathcal{B}}
\nc\bse{\mathcal{E}}
\nc\flz{\mathcal{Z}}
\nc\flm{\mathcal{M}}
\nc\fly{\mathcal{Y}}
\nc\bdelta{\mathbf{d}}
\nc\ve{\varepsilon}
\nc{\tl}{\tilde}
\nc{\ee}{{\mathbb{E}\,}}
\nc{\cee}[3]{{\ee #3[#1\,|\,#2 #3]}}
\nc{\norma}[1]{{\| #1 \|}}
\nc{\ind}[1]{1_{\{#1\}}\,}
\nc{\ilskal}[2]{{\langle {#1} , {#2} \rangle}}
\nc{\setcond}{: \exs}
\nc{\supp}{\mathop{\rm supp} \nolimits}
\nc{\conv}{\mathop{\rm conv} \nolimits}
\nc{\essinf}{\mathop{\rm ess\, inf} \limits}
\nc{\esssup}{\mathop{\rm ess\, sup} \limits}
\nc{\vvect}[2]{{\begin{pmatrix} #1 \\ #2 \end{pmatrix}}}
\nc{\vvvect}[3]{{\begin{pmatrix} #1 \\ #2 \\ #3 \end{pmatrix}}}
\nc\cont[1]{C^{\mathcal #1}}
\nc\bb{E^B}

\begin{document}

\author{Jan~Palczewski}
\address{Faculty of Mathematics, University of Warsaw, Banacha 2, 02-097 Warszawa, Poland and School of Mathematics, University of Leeds, Leeds LS2 9JT, UK}
\email{J.Palczewski@mimuw.edu.pl}

\author{\L ukasz Stettner}
\address{Institute of Mathematics, Polish Academy of Sciences, \'{S}niadeckich 8, 00-956 Warszawa, Poland and  Academy of Finance, Warszawa, Poland}
\email{stettner@impan.gov.pl}

\title{Stopping of functionals with discontinuity at the boundary of an open set
}
\thanks{Research of both authors supported in part by MNiSzW Grant no. NN 201 371836.}

\date{$\,$Date: 2011-04-28 13:26:31 $\,$}

\maketitle

\begin{abstract}
We explore properties of the value function and existence of optimal stopping times for functionals with discontinuities related to the boundary of an open (possibly unbounded) set $\mathcal{O}$. The stopping horizon is either random, equal to the first exit from the set $\mathcal{O}$, or fixed: finite or infinite. The payoff function is continuous with a possible jump at the boundary of $\mathcal{O}$. Using a generalization of the penalty method we derive a numerical algorithm for approximation of the value function for general Feller-Markov processes and show existence of optimal or $\varepsilon$-optimal stopping times.

\end{abstract}

\setcounter{secnumdepth}{2}

\section{Introduction}
The problem of optimal stopping of Markov processes has received continuous attention for last fourty years and produced diverse approaches for its solution. Foundations and general existence results can be found, e.g., in Bismut and Skalli \cite{bismut1977}, El Karoui \cite{elkarui1981}, El Karoui et al. \cite{elkarui1982}, Fakeev \cite{fakeev}, and Mertens \cite{mertens}. From 1980s functional analytic methods gave way to a more explicit approach initiated by Bensoussan and Lions \cite{Bens-Lions}: value function was characterized as a solution to a variational inequality, which could be solved analytically or numerically. The main limitation of this method is the requirement of a particular differential form of the generator of the underlying Markov process. This paper belongs to another strand of literature which initially aimed at studying smoothness of the value function but also provides a different approach for the numerical approximation to the value function for a more general class of Markov processes (see Zabczyk \cite{Z} for a survey). These methods are not constrained by the form of generators and the development of the theory of PDEs. Specifically, we build on the penalty method introduced by Robin \cite{robin} and generalized by Stettner and Zabczyk \cite{SZ} (see also \cite{stettner}), which originates in ideas developed for partial differential equations but follows a purely probabilistic route. Of interest to numerical methods discussed in this paper is also a time-discretization technique explored by Mackevicius \cite{mac} and further applied by Kushner and Dupuis \cite{KD} for numerical algorithms; see also Palczewski and Stettner \cite{palczewski2008} for its application to stopping of time-discontinuous functionals.

We assume that the state of the world is described by a standard Markov process $\big(X(t)\big)$ defined on a locally compact separable space $E$ endowed with a metric $\rho$ with respect to which every closed ball is compact (see the Appendix for the definition and properties of standard Markov processes). The Borel $\sigma$-algebra on $E$ is denoted by $\bse$. The process $\big(X(t)\big)$ satisfies the weak Feller property:
$$
P_t\, \mathcal{C}_0 \subseteq \mathcal{C}_0,
$$
where $\mathcal C_0$ is the space of continuous bounded functions $E \to \er$ vanishing in infinity, and $P_t$ is the transition semigroup of the process $\big(X(t)\big)$, i.e., $P_t h(x) = \ee^x \left\{h\big(X(t)\big)\right\}$ for any bounded measurable $h:E \to \er$.

Let $\lo \subset E$ be an open set and $\tau_\lo = \inf\{t: X(t) \notin \lo\}$ -- the first exit time from $\lo$. We study maximization of several classes of functionals:
\begin{enumerate}
\item Stopping is allowed up to time $\tau_\lo$. The payoff is described by a function $G$ before $\tau_\lo$ and by a function $H$ at $\tau_\lo$:
\begin{equation}
\label{eqn:simple_vf}
\begin{aligned}
J(s, x, \tau) &= \ee^x\Big\{\int_0^{\tau \wedge \tau_\lo} e^{-\alpha u} f\big(s+u,X(u)\big) du\\
&\mop{35} + \ind{\tau<\tau_\lo}e^{-\alpha \tau} G\big(s+\tau,X(\tau)\big)+\ind{\tau\geq \tau_\lo} e^{-\alpha \tau_\lo} H\big(s + \tau_\lo, X(\tau_\lo)\big)\Big\},
\end{aligned}
\end{equation}
where $(s, x) \in [0, \infty)\times E$, $\alpha > 0$,  $\tau \ge 0$ and $f, G, H:[0, \infty) \times E \to \er$ are continuous bounded functions.
\item Stopping is allowed up to time $\tau_\lo$ and the payoff is given by a function $F:[0, \infty) \times E \to \er$ which is continuous on $[0, \infty) \times \lo$ and possibly discontinuous in the space variable on $[0, \infty) \times \lo^c$:
\begin{equation}
\label{eqn:functinal_2}
J(s, x, \tau) = \ee^x\Big\{\int_0^{\tau \wedge \tau_\lo} e^{-\alpha u} f\big(s+u,X(u)\big) du+ e^{-\alpha (\tau \wedge \tau_\lo)} F\big(s+(\tau \wedge \tau_\lo),X(\tau \wedge \tau_\lo)\big)\Big\}.
\end{equation}
This, in particular, covers a complementary problem to \eqref{eqn:simple_vf}: with $F$  continuous on $[0, \infty) \times \bar \lo$ and on $[0, \infty) \times \bar \lo^c$ with a possible jump at the boundary $[0, \infty)\times\partial \lo$.

\item Stopping is unconstrained (infinite horizon, $T=\infty$) or constrained by a constant $T$ (finite horizon) with the following functional:
\begin{equation}
\label{eqn:functinal_3}
J(s, x, \tau) = \ee^x\Big\{\int_0^{\tau \wedge (T-s)} e^{-\alpha u} f\big(s+u,X(u)\big) du+ e^{-\alpha (\tau \wedge T)} F\big((s+\tau) \wedge T,X(\tau \wedge (T-s))\big)\Big\},
\end{equation}
where the payoff function $F$ is continuous apart from a possible discontinuity on $[0, \infty) \times \partial \lo$.
\end{enumerate}

Optimal stopping problems of the first type were studied by Bensoussan and Lions \cite{Bens-Lions} for non-degenerate diffusion processes under assumptions that $G \le H$ and $\lo$ is bounded with a smooth boundary $\partial \lo$. They used penalization techniques similar  to ours but applied them on the level of variational inequalities. Generalizations were attempted by many authors in two directions: to extend the class of processes for which this approach applies and to relax assumptions on the functional; see, e.g., Menaldi \cite{menaldi} for the removal of restrictions on degeneracy of the diffusion, and Fleming, Soner \cite{flemingsoner} for relaxation of many assumptions regarding the functional and the coefficients of the diffusion via viscosity solutions approach. Functionals of the third type recently gained a lot of attention. Lamberton \cite{lamberton} obtained continuity and variational characterization of the value function for stopping of one-dimensional diffusions with bounded and Borel-measurable payoff function $F$. His result, however, cannot be extended to multidimensional diffusions. Bassan and Ceci studied stopping of  semi-continuous payoff functions $F$ for diffusions and certain jump-diffusions in one dimension (\cite{bassan2002a, bassan2002b}). They proved that value function for a functional with lower/upper semi-continuous function $F$ is lower/upper semi-continuous. The existence of optimal stopping times was also shown but without an explicit construction.

This paper complements existing theory in two aspects. Firstly, it provides results for a far larger family of Markov processes (in particular, in dimensions higher than $1$) and enables numerical treatment of the value function. Secondly, it relaxes constraints on the region $\lo$, which can be unbounded and with non-smooth boundary. Our main assumption is that the mapping $x \mapsto \ee^x \{ \ind{\tau_\lo < t} h(X_t)\}$ is continuous for any $t > 0$ and a continuous bounded function $h$. This assumption is non-restrictive as we show in Section \ref{sec:conditions_for_A1}. It is usually satisfied by solutions to nondegenerate stochastic differential equations driven by Brownian or Levy noise. Consequently our results, based on probabilistic arguments, provide regularity of solutions to differential or integrodifferential variational inequalities, related to appropriate stopping problems, with various types of discontinuity.

In our approach, the value function is approximated by a sequence of penalized value functions which are unique fixed points of contraction operators. These operators do not involve stopping or any other type of control, which makes them easier to compute numerically. Moreover, a discrete approximation of the state space can be used because we prove that the penalized functions are continuous.

The remaining of the paper is organized in the following way. Section \ref{sec:penalty} introduces the penalty method for functionals of the first type. The following section explores the properties of the value function, in particular, its behaviour on the boundary $\partial \lo$. In Section \ref{sec:cont_w} main results on optimal and $\ve$-optimal stopping and the convergence of penalized value functions are obtained. Sufficient conditions for the main assumption \ref{ass:a1} are formulated in Section \ref{sec:conditions_for_A1}. Functionals of the second type are studied in Section \ref{sec:general_func}. Section \ref{sec:infinite} extends these results to functionals of the third type with infinite time horizon. A finite time horizon setting is studied in Section \ref{sec:finite_time_func}. Important properties of Feller processes are listed in the Appendix.

\section{Penalty method}\label{sec:penalty}

We solve the stopping problem \eqref{eqn:simple_vf} using the penalty method introduced by Robin \cite{robin} and generalized by Stettner and Zabczyk \cite{SZ}. For $\beta > 0$ consider a penalized equation
\begin{equation}\label{eqn:def_w_beta}
\begin{aligned}
w^\beta (s, x) &= \ee^x\Big\{\int_0^{\tau_\lo} e^{-\alpha u} \Big[ f\big(s+u,X(u)\big)  + \beta \Big( G\big(s+u,X(u)\big) - w^\beta\big(s+u,X(u)\big) \Big)^+ \Big] du\\
&\mop{35} +e^{-\alpha \tau_\lo} H\big(s + \tau_\lo, X(\tau_\lo)\big)\Big\}.
\end{aligned}
\end{equation}

\begin{lemma}\label{lem:transf}
Assume that $g$ and $h$ are bounded functions and $\alpha > 0$. For any bounded progressively measurable process $(b(t))$, the following formulae
\begin{align}
z(s,x) &= \ee^x \Big\{\int_0^{\tau_\lo} e^{-\alpha u}g(s+u,x(u))du
+ e^{-\alpha \tau_\lo} h(s+\tau_\lo,x(\tau_\lo)) \Big\}, \label{eq:4}\\[10pt]
z(s,x)&= \ee^x \Big\{\int_0^{\tau_\lo} e^{-\alpha u - \int_0^u
b(t)dt}\big[g(s+u,x(u)) +b(u)z(s+u,x(u))\big]du \nonumber\\
&\mop{35}+ e^{-\alpha \tau_\lo - \int_0^{\tau_\lo} b(t)dt}h(s+\tau_\lo,x(\tau_\lo))\Big\} \label{eq:5}
\end{align}
are equivalent in the following sense: $z$ defined in \eqref{eq:4} is a solution to \eqref{eq:5};
and any solution to \eqref{eq:5} is of the form \eqref{eq:4}.
\end{lemma}
\begin{proof}
We use similar arguments as in Lemma 1 of \cite{stettner}. The only difference is that now we have $\tau_\lo$
instead of the deterministic time $T-s$.
\end{proof}
Using this lemma, in a similar way as in Proposition 1 of \cite{stettner}, we show
\begin{lemma}\label{lem:solv_w_beta}
There is exactly one bounded measurable function $w^\beta$ that satisfies \eqref{eqn:def_w_beta}.
\end{lemma}
\begin{proof}
By Lemma \ref{lem:transf} the penalized function $w^\beta$ can be equivalently written as
\begin{equation}\label{eqn:def_w_beta equiv}
\begin{aligned}
w^\beta (s, x) &= \ee^x\Big\{\int_0^{\tau_\lo} e^{-(\alpha+\beta) u} \Big[ f\big(s+u,X(u)\big)\\
&\mop{80} + \beta \Big( G\big(s+u,X(u)\big) - w^\beta\big(s+u,X(u)\big) \Big)^+ +\beta w^\beta\big(s+u,X(u)\big)  \Big] du\\
&\mop{30} +e^{-(\alpha+\beta) \tau_\lo} H\big(s + \tau_\lo, X(\tau_\lo)\big)\Big\}.
\end{aligned}
\end{equation}
Hence, $w^\beta$ is a fixed point of the operator $\mathcal T$ defined for measurable bounded functions $\phi$ as follows:
\begin{align*}
\mathcal T \phi(s,x) &= \ee^x\Big\{\int_0^{\tau_\lo} e^{-(\alpha+\beta) u} \Big[ f\big(s+u,X(u)\big)\\
&\mop{80} + \beta \Big( G\big(s+u,X(u)\big) - \phi\big(s+u,X(u)\big) \Big)^+ +\beta \phi\big(s+u,X(u)\big) \Big] du\\
&\mop{30} +e^{-(\alpha+\beta) \tau_\lo} H\big(s + \tau_\lo, X(\tau_\lo)\big)\Big\}.
\end{align*}
This operator is a contraction on the space of measurable bounded functions for any $\beta > 0$. Indeed, $\mathcal T \phi$ is identically equal to $H$ on $[0, \infty) \times \lo^c$, whereas the contraction property on $[0, \infty) \times \lo$ follows from the estimate
$$
\mathcal T \phi_1 - \mathcal T \phi_2 \le \frac{\beta}{\alpha + \beta} \|\phi_1 - \phi_2\|_\infty.
$$
This implies that $w^\beta$ is a unique fixed point of $\mathcal T$.
\end{proof}

We make the following assumption
\begin{assumption}
 \item[(A1)] \label{ass:a1}
The stopped semigroup $P^{\tau_\lo}_t h(x) = \ee^x \{ \ind{t < \tau_\lo} h(X(t))\}$ maps the space of continuous bounded functions into itself.
\end{assumption}
The following three lemmas prove continuity results which, in particular, will be used to show that $w^\beta$ is continuous.
\begin{lemma}\label{lem:joint Feller}
Under (A1), for a continuous bounded function $h: [0,\infty)\times E \to (-\infty, \infty)$ the mapping
$$
(s,x)\mapsto P^{\tau_\lo}_t h(s,x) := \ee^x \{ \ind{t < \tau_\lo} h(s+t,X(t))\}$$
is continuous.
\end{lemma}
\begin{proof}
Let $(s_n,x_n)\to (s,x)$. By Proposition \ref{prop:01} for a given $\ve>0$ there is a compact set $K\subset E$ such that
$$\prob^{x_n} \left\{\exists_{u\in [0,s+t+1]} X(u)\notin K \right\}\leq \ve . $$
For $n$ large enough, i.e., such that $|s - s_n| \le 1$, we have
\begin{align*}
&\big|P^{\tau_\lo}_t h(s_n,x_n)-P^{\tau_\lo}_t h(s,x)\big|\\
&\mop{0}\le
\Big| \ee^{x_n} \{\ind{t < \tau_\lo} h(s_n+t,X(t))\}-\ee^{x_n} \{\ind{t < \tau_\lo} h(s+t,X(t))\} \Big| \\
&\mop{13}+
\Big| \ee^{x_n} \{\ind{t < \tau_\lo} h(s+t,X(t))\}-\ee^{x} \{\ind{t < \tau_\lo} h(s+t,X(t))\} \Big|\\
&\mop{0}\le
\ve \|h\|+ \Big| \ee^{x_n} \{\ind{t < \tau_\lo} \ind{X(t)\in K} h(s_n+t,X(t))-h(s+t,X(t))\}\Big| \\
&\mop{13}+
\Big| \ee^{x_n} \{\ind{t < \tau_\lo} h(s+t,X(t))\}-\ee^{x} \{\ind{t < \tau_\lo} h(s+t,X(t))\}\Big|\\
&\mop{0}= \ve \|h\|+a_n+b_n.
\end{align*}
The sequence $a_n$ converges to $0$ by uniform continuity of $h$ is on $[0,s+t+1]\times K$. Assumption (A1) implies $b_n\to 0$, which completes the proof.
\end{proof}

\begin{lemma}\label{lem:cont_stopped_potential}
Under assumption (A1) the mapping
$$
(s,x)\mapsto \ee^x \Big\{ \int_0^{\tau_\lo} e^{-\gamma u} h(s+u, X(u)) du \Big\}
$$
is continuous for any function $h \in \mathcal C([0, \infty)\times E)$ and $\gamma > 0$.
\end{lemma}
\begin{proof}
Fubini's theorem implies
$$
\ee^x \Big\{ \int_0^{\tau_\lo} e^{-\gamma u} h(s+u, X(u)) du \Big\} = \int_0^\infty e^{-\gamma u} \varphi (s, u, x) du,
$$
where $\varphi(s, u,x) = \ee^x \{ \ind{u < \tau_\lo} h(s + u, X(u) \}$. This function is continuous in $(s,x)$ for any fixed $u \ge 0$ by Lemma \ref{lem:joint Feller}. Dominated convergence theorem concludes.
\end{proof}

\begin{lemma}\label{lem:formerA2}
Under (A1) for $\alpha > 0$ and a continuous bounded function $h:[0,\infty)\times E \mapsto (-\infty,\infty)$ the mapping
$$
(s,x)\mapsto \ee^x \Big\{e^{-\alpha \tau_\lo} h(s+\tau_\lo, X(\tau_\lo)) \Big\}
$$
is continuous.
\end{lemma}
\begin{proof}
Assume first that
\begin{equation}\label{eqn:decompos1}
h(s,x) = \ee^x \Big\{ \int_0^\infty e^{-\alpha u} \tl h \big(s+u, X(u)\big) du \Big\}
\end{equation}
for a continuous bounded function $\tl h$. Using this decomposition we write
\begin{align*}
H(s,x) &= \ee^x \Big\{e^{-\alpha \tau_\lo} h(s+\tau_\lo, X(\tau_\lo)) \Big\}
=
\ee^x \Big\{ \int_{\tau_\lo}^\infty e^{-\alpha u} \tl h \big(s+u, X(u)\big) du \Big\}\\
&=
h(s,x) - \ee^x \Big\{ \int_0^{\tau_\lo} e^{-\alpha u} \tl h \big(s+u, X(u)\big) du \Big\}.
\end{align*}
Hence, $H$ is continuous by Lemma \ref{lem:cont_stopped_potential}.

By the weak Feller property of $(X(t))$ functions of the form \eqref{eqn:decompos1} are dense in $C_0 ([0, \infty) \times E)$ (see Lemma 3.1.6 in \cite{dynkin}).
Hence, $H$ is continuous for $h$ in $C_0$. The extension of this result to continuous bounded functions $h$ uses Proposition \ref{prop:01} in the appendix. Fix a compact set $K \subseteq E$ and $S \ge 0$. For any $T, \ve > 0$ there is a compact set $L \subseteq E$ such that
$$
\prob^x\big( X(t) \notin L \text{ for some $t \in [0, T]$} \big) < \ve, \qquad \forall\ x \in K.
$$
Define $r(s,x) = e^{-\rho(x, L) - (s - (S + T))^+} h(s,x)$, where $\rho(x, L)$ denotes the distance of $x$ from the set $L$. Such $r$ is in $C_0([0, \infty) \times E)$ and by preceding results $R(s,x) = \ee^x \Big\{e^{-\alpha \tau_\lo} r(s+\tau_\lo, X(\tau_\lo)) \Big\}$ is continuous. By definition $r \equiv h$ on $[0, S+T] \times L$. Let $A = \{X(t) \notin L \text{ for some $t \in [0, T]$} \}$.  The distance of $R$ and $H$ is bounded in the following way:
\begin{align*}
\|H(s,x) - R(s,x)\|
&= \ee^x \big\{ e^{-\alpha \tau_\lo} (h-r) (s + \tau_\lo, X(\tau_\lo)) \big\} \\
&= \ee^x \big\{ e^{-\alpha \tau_\lo \wedge T} \ind{A^c} (h-r)(s + \tau_\lo \wedge T, X(\tau_\lo \wedge T)) \big\}\\
&\mop{13}+ \ee^x \big\{ \ind{A} e^{-\alpha \tau_\lo} (h-r) (s + \tau_\lo, X(\tau_\lo)) \big\}\\
&\mop{13}+ \ee^x \Big\{ \ind{A^c} \Big(e^{-\alpha \tau_\lo} (h-r) (s + \tau_\lo, X(\tau_\lo)) \\
& \mop{46}- e^{-\alpha \tau_\lo \wedge T}(h-r)(s + \tau_\lo \wedge T, X(\tau_\lo \wedge T)) \Big)\Big\}\\
&\le
0 + \|h-r\| \ve + e^{-\alpha T} 2\|h-r\|.
\end{align*}
This implies $\|H(s,x) - R(s,x)\| \le 2 \|h\| (\ve + 2e^{-\alpha T})$ for $(s,x) \in [0, S] \times K$. Since $T$ and $\ve$ are arbitrary this implies continuity of $H$ on $[0, S] \times K$. Hence, $H$ is continuous on its whole domain by the arbitrariness of $S, K$.
\end{proof}

\begin{cor}\label{cor:cont_wbeta}
Under (A1), the unique bounded solution $w^\beta$ of \eqref{eqn:def_w_beta} is continuous.
\end{cor}
\begin{proof}
Lemmas \ref{lem:cont_stopped_potential} and \ref{lem:formerA2} imply that the operator $\mathcal T$ introduced in the proof of Lemma \ref{lem:solv_w_beta} maps the space of continuous bounded functions into itself. Since this operator is a contraction on the space of bounded measurable functions it is a contraction on the space of continous bounded functions. This implies that $w^\beta$ as a unique fixed point is continuous.
\end{proof}

To establish convergence of $w^\beta$ to $w$ as $\beta \to \infty$ we introduce two additional representations of $w^\beta$.
\begin{lemma} \label{lem:wbeta_form_tau}
Under assumption (A1), the function $w^\beta$ has the following equivalent representation:
\begin{equation}\label{eqn:wbeta_form_tau}
\begin{aligned}
w^\beta (s, x) = \sup_\tau \Big\{ J(s, x, \tau) - \ee^x \big\{\ind{\tau < \tau_\lo} e^{-\alpha \tau}
\big( G - w^\beta \big)^+\big(s+\tau,X(\tau)\big) \big\}\Big\}.
\end{aligned}
\end{equation}
\end{lemma}
\begin{proof}
Markov property implies that for any stopping time $\sigma$ the following equality is satisfied:
\begin{align*}
&w^\beta(s,x)= \ee^x \Big\{\int_0^{\tau_\lo \wedge \sigma} e^{-\alpha u}\Big[ f\big(s+u,X(u)\big)+ \beta \Big( G\big(s+u,X(u)\big) - w^\beta\big(s+u,X(u)\big) \Big)^+ \Big]du\\
&\mop{75}+ e^{-\alpha (\tau_\lo \wedge \sigma)}w^\beta(s+\tau_\lo \wedge \sigma,X(\tau_\lo \wedge \sigma))\Big\}.
\end{align*}
This gives the lower bound:
\begin{multline}\label{eqn:wbeta_stop1}
w^\beta(s,x) \ge\ee^x \Big\{\int_0^{\tau_\lo \wedge \sigma} e^{-\alpha u} f\big(s+u,X(u)\big) du
+ e^{-\alpha (\tau_\lo \wedge \sigma)}w^\beta(s+\tau_\lo \wedge \sigma,X(\tau_\lo \wedge \sigma))\Big\}.
\end{multline}
Further,
\begin{equation}\label{eqn:wbeta_stop2}
\begin{aligned}
w^\beta(s,x) &\ge \ee^x \Big\{\int_0^{\tau_\lo \wedge \sigma} e^{-\alpha u} f\big(s+u,X(u)\big) du\\
&\mop{30}+ \ind{\sigma < \tau_\lo} e^{-\alpha \sigma} \big[G - (G - w^\beta)^+\big]\big(s+\sigma, X(\sigma)\big)+ \ind{\sigma \ge \tau_\lo} e^{-\alpha \tau_\lo}H(s+\tau_\lo,X(\tau_\lo))\Big\},
\end{aligned}
\end{equation}
because $w^\beta = H$ on $\lo^c$ and $G - (G-w^\beta)^+ \le w^\beta$. Define a stopping time
$$
\sigma^* = \inf\{u \ge 0: w^\beta (s+u, X(u)) \le G(s+u, X(u)) \}.
$$
Due to the continuity of $G$ and $w^\beta$ (see Corollary \ref{cor:cont_wbeta}) we have
$$
\ind{\sigma^* < \tau_\lo} w^\beta (s+\sigma^*, X(\sigma^*)) \le \ind{\sigma^* < \tau_\lo} G(s+\sigma^*, X(\sigma^*)).
$$
This implies that for $\sigma^*$ the inequalities in \eqref{eqn:wbeta_stop1} and \eqref{eqn:wbeta_stop2} become equalities and \eqref{eqn:wbeta_form_tau} follows easily.
\end{proof}

\begin{lemma} \label{lem:wbeta_form_mbeta}
The function $w^\beta$ has the following equivalent representation:
\begin{equation}\label{eqn:wbeta_form_mbeta}
\begin{aligned}
w^{\beta}(s,x)
&= \sup_{b\in M_\beta} \ee^x\Big\{\int_0^{\tau_\lo}e^{-\alpha u - \int_0^{u}b(t)dt} \Big[f\big(s+u,X(u)\big)+ b(u)G\big(s+u,X(u)\big)\Big]du\\
&\mop{50}+e^{-\alpha \tau_\lo - \int_0^{\tau_\lo} b(t)dt} H\big(s+\tau_\lo,X(\tau_\lo)\big)\Big\},
\end{aligned}
\end{equation}
where $M_\beta$ is the class of progressively measurable processes with values in $[0,\beta]$.
\end{lemma}
\begin{proof}
By Lemma \ref{lem:transf} the function $w^\beta$ has the following equivalent formulation:
\begin{align*}
w^\beta(s,x) &= \ee^x \Big\{\int_0^{\tau_\lo} e^{-\alpha u - \int_0^u
b(t)dt}\Big[ f\big(s+u,X(u)\big)\\
&\mop{25} + \beta \Big( G\big(s+u,X(u)\big) - w^\beta\big(s+u,X(u)\big) \Big)^++ b(u)w^\beta(s+u,x(u))\Big]du\\
&\mop{25}+ e^{-\alpha \tau_\lo - \int_0^{\tau_\lo} b(t)dt}h(s+\tau_\lo,x(\tau_\lo))\Big\}
\end{align*}
for any progressively measurable process $b(t)$ with values in $[0, \beta]$. Since $b(t) \le \beta$ we have
$$
\beta \Big( G\big(s+u,X(u)\big) - w^\beta\big(s+u,X(u)\big) \Big)^++ b(u)w^\beta(s+u,X(u))
\ge b(u) G\big(s+u, X(u) \big),
$$
which implies
\begin{align*}
&w^\beta(s,x) \ge \ee^x \Big\{\int_0^{\tau_\lo} e^{-\alpha u - \int_0^u
b(t)dt}\Big[ f\big(s+u,X(u)\big)+ b(u) G\big(s+u, X(u) \big) \Big] du \\
&\mop{58}+ e^{-\alpha \tau_\lo - \int_0^{\tau_\lo} b(t)dt}h(s+\tau_\lo,X(\tau_\lo))\Big\}.
\end{align*}
This is an equality for $b(t)$ given by
$$
b(t) = \begin{cases}
        \beta, & G\big(s+t, X(t)\big) \ge w^\beta \big(s+t, X(t)\big),\\
	0, & \text{otherwise}.
       \end{cases}
$$
Hence, formula \eqref{eqn:wbeta_form_mbeta} is proved.
\end{proof}

\begin{prop}\label{prop:pointwise}
Under (A1), the functions $w^\beta(s,x)$ increase pointwise to $w(s,x)$ as $\beta \to \infty$.
\end{prop}
\begin{proof}
Equation \eqref{eqn:wbeta_form_mbeta} implies that the functions $w^\beta(s,x)$ are increasing in $\beta$. Hence the limit $w^\infty(s,x) = \lim_{\beta \to \infty} w^\beta(s,x)$ exists. By \eqref{eqn:wbeta_form_tau} we  have $w^\beta \le w$ and, therefore, $w^\infty\leq w$. To prove that $w^\infty = w$ we first show that $w^\infty\geq G$. Let $x\in \lo$ and, for $\eta > 0$, put $b^\eta(u)=\ind{u \le \eta} \beta$. Then by \eqref{eqn:wbeta_form_mbeta} we have
\begin{align*}
w^{\beta}(s,x)
&\ge \ee^x\Big\{\int_0^{\tau_\lo}e^{-\alpha u - \int_0^{u}b^\eta(t)dt} f\big(s+u,X(u)\big) du\\
&\mop{30}+ \int_0^{\tau_\lo \wedge \eta}e^{-(\alpha+\beta)u}\beta G\big(s+u,X(u)\big)du+e^{-\alpha \tau_\lo - \int_0^{\tau_\lo} b^\eta(t)dt} H\big(\tau_\lo,X(\tau_\lo)\big)\Big\} \\
&= \ee^x \big\{ (I) + (II) + (III) \big\}.
\end{align*}
Letting $\beta \to \infty$ we can make $(I)$ and $(III)$ arbitrarily small and for sufficiently small $\eta$ and large $\beta$ the term $(II)$ is arbitrarily close to $G(s,x)$. Dominated Convergence Theorem implies $w^\infty(s,x) \ge G(s,x)$.

From (\ref{eqn:wbeta_form_tau}) for any stopping time $\tau$ we have
$$
w^\beta (s, x) \ge J(s, x, \tau) - \ee^x \big\{\ind{\tau < \tau_\lo} e^{-\alpha \tau}
\big( G - w^\beta \big)^+\big(s+\tau,X(\tau)\big) \big\}.
$$
By letting $\beta \to \infty$ we obtain
\begin{equation}
w^\infty(s,x)\geq J(s, x, \tau),
\end{equation}
because $\lim_{\beta \to \infty} (G - w^\beta)^+ (s,x) = 0$ for $x \in \lo$. Since $\tau$ is arbitrary we conclude that $w^\infty(s,x)=w(s,x)$ for $x\in \lo$. For $x\in E\setminus \lo$ we have $w^\beta(s,x)=H(s,x)=w(s,x)$.
\end{proof}

\begin{cor}\label{cor:lower_cont_w}
Under (A1), the value function $w$ is lowersemicontinuous. Moreover, if $w$ is continuous then $w^\beta$ approaches $w$ uniformly on compact sets.
\end{cor}
\begin{proof}
The semicontinuity of $w$ follows from Corollary \ref{cor:cont_wbeta} and Proposition \ref{prop:pointwise}. Dini's theorem implies uniform convergence on compact sets if $w$ is continuous.
\end{proof}

\section{Properties of the value function $w$}\label{sec:boundary}

In this section we explore the properties of the value function, in particular, its behaviour on the boundary of $\lo$.

\begin{thm}\label{thm:limit}
Under (A1),  for $x \in \partial\lo$ we have
\begin{equation}\label{eq:limit}
\lim_{y\to x, y \in \lo} w(s,y)=G \vee H(s,x).
\end{equation}
\end{thm}

The proof of this theorem consists of several steps which are of interest on their own. They are formulated and proved as separate results below.

It is clear that $w \ge G$ on $\lo$ and $w = H$ on the complement of $\lo$. It is therefore natural to expect a discontinuity at the boundary of $\lo$ if $G > H$. The following proposition shows that this discontinuity is constrained to the minimum: the absolute value of the difference between $G$ and $H$.

\begin{prop}\label{prop:boundary_downjump}
Assume (A1) and $G \ge H$. For any $x \in \partial \lo$ we have
\begin{equation}\label{eqn:prop:bd_dj}
\lim_{y \to x, \ y \in \lo} w(s, y) = G(s,x)
\end{equation}
and the convergence is uniform in $s$ and $x$ from compact sets.
\end{prop}
\begin{proof}
Since $w(s,x) \ge G(s,x)$ for $x \in \lo$ and $G$ is continuous we obtain that
$\liminf_{y \to x, \ y \in \lo} w(s, y)\linebreak \ge G(s,x)$. In the remaining part of the proof we show that $\limsup_{y \to x, \ y \in \lo} w(s, y) \le G(s,x)$, which implies that the limit in \eqref{eqn:prop:bd_dj} exists and equals $G(s,x)$.

Fix a compact set $K \subseteq E$, $T > 0$ and $\ve > 0$. First we make preparatory steps. By Proposition \ref{prop:01} in the Appendix, there is a compact set $L \subseteq E$ such that
\begin{equation}\label{eqn:prop_bound_1}
\sup_{x \in K} \prob^x \big( \exists\ t\in [0, T+1]\ X(t) \notin L \big) \le \ve.
\end{equation}
The extension of the time interval by one unit to $[0, T+1]$ is required to allow the initial time $s$ to be in $[0, T]$ and leave time for the process $(X(t))$ to evolve. Notice that below $\delta$ and $\eta$ are both bounded by $1$.

Let $\delta \in (0, 1)$ be such that for $(s, x) \in [0, T] \times B(L, \delta)$, $y \in L$, $\|x - y\| \le \delta$ and $t \in [0, \delta]$
\begin{equation}\label{eqn:prop_bound_2}
| G(s, x) - G(s + t, y)| \le \ve,\\
\end{equation}
Proposition \ref{prop:fellercont} implies that there is $\eta > 0$, which, for convenience, is bounded by $\delta \wedge \ve$, such that
\begin{equation}\label{eqn:short_time_cont}
\sup_{x \in L}\ \sup_{t \le \eta} \prob^x \big( X(t) \notin B(x, \delta) \big) \le \ve.
\end{equation}

Fix $x \in \partial \lo \cap K$ and $s \in [0, T]$. For any $y \in \lo \cap K$ we have
\begin{align*}
w(s,y) &= \sup_{\tau} J(s,y, \tau) \\
&\le
\sup_\tau \ee^y \Big\{ \int_0^{\tau \wedge \tau_\lo} e^{-\alpha u} f(s+u, X(u)) su +  e^{-\alpha (\tau \wedge \tau_\lo)}G(s + \tau \wedge \tau_\lo, X(\tau \wedge \tau_\lo)) \Big\}\\
&\le
\prob^y( \tau_\lo > \eta) \big(\frac{\|f\|}{\alpha} + \|G\|\big) + \prob^y (\tau_\lo \le \eta) \big( \eta \|f\| + G(s,y) \big)\\
&\mop{11}+
\sup_\tau \ee^y \big\{ \ind{\tau_\lo \le \eta}\big| e^{-\alpha (\tau \wedge \tau_\lo)} G(s + \tau \wedge \tau_\lo, X(\tau \wedge \tau_\lo)) - G(s,y) \big| \big\}.
\end{align*}
Consider the last term. For any stopping time $\tau$ we have
\begin{align*}
&\ee^y \big\{ \ind{\tau_\lo \le \eta}\big| e^{-\alpha (\tau \wedge \tau_\lo)} G(s + \tau \wedge \tau_\lo, X(\tau \wedge \tau_\lo)) - G(s,x) \big| \big\}\\
& \le  \|G\|(1-e^{-\alpha \eta}) +
\ee^y \big\{\big| G(s + \tau \wedge \tau_\lo \wedge \eta, X(\tau \wedge \tau_\lo \wedge \eta)) - G(s,x) \big| \big\}\\
&\le  \alpha \eta \|G\| +
\ee^y \big\{\big| G(s + \tau \wedge \tau_\lo \wedge \eta, X(\tau \wedge \tau_\lo \wedge \eta)) - G(s+\eta,X(\eta)) \big| \big\}\\
&\mop{13} +\ee^y \big\{\big| G(s + \eta, X(\eta)) - G(s,y) \big| \big\}\\
&= (I) + (II) + (III).
\end{align*}
The first term is bounded by $\alpha \ve \|G\|$. The estimate of the second term requires conditioning on $X(\tau \wedge \tau_\lo \wedge \eta)$, the use of the strong Markov property and inequalities (\ref{eqn:prop_bound_1}), (\ref{eqn:short_time_cont}):
\begin{align*}
&\ee^y \big\{\big| G(s + \tau \wedge \tau_\lo \wedge \eta, X(\tau \wedge \tau_\lo \wedge \eta)) - G(s+\eta,X(\eta)) \big| \big\}\\
&= \ee^y \Big\{ \ee^{X(\tau \wedge \tau_\lo \wedge \eta)} \Big\{\big| G(s + \tau \wedge \tau_\lo \wedge \eta, X(0))- G(s+\eta,X(\eta-\tau \wedge \tau_\lo \wedge \eta)) \big| \Big\} \Big\}\\
&\le
2 \|G\|\ \prob^y \{ \exists s \in [0, \eta]\  X(s) \notin L \}\\
&\mop{13}+ 2 \|G\|\  \prob^y \big\{\forall s \in [0, \eta]\ X(s) \in L \quad \text{and} \quad X(\eta) \notin B\big(X(\tau \wedge \tau_\lo \wedge \eta), \delta\big)\big\} + \ve\\
&\le 2\|G\| \ve + 2 \|G\| \ve + \ve = \ve (1 + 4 \|G\|).
\end{align*}
Term $(III)$ is estimated similarly knowing that $y$ is in $L$ by assumption: $(III) \le \ve (1 + 2 \|G\|)$. Combining these estimates  we obtain
\begin{align*}
w(s,y) &\le h_\eta(y) \big( \frac{\|f\|}{\alpha} + \|G\| \big) + (1 - h_\eta(y)) \big( \eta \|f\| + G(s,y)\big) + \ve \big(2 + (6+\alpha) \|G\|\big)\\
&\le G(s,y) + h_\eta(y)  \big( \frac{\|f\|}{\alpha} + 2\|G\| \big) + \eta \|f\| + \ve \big(2 + (6+\alpha) \|G\|\big),
\end{align*}
where $h_{\eta}(y) = \prob^y \{ \tau_\lo > \eta)$. Assumption (A1) implies that $h_\eta$ is continuous on $E$. Clearly, $h_\eta(x) = 0$. Hence,
$$
\limsup_{y \to x,\ y \in \lo} w(s,y) \le G(s,x) 
+ \eta \|f\| + \ve \big(2 + (6+\alpha) \|G\|\big)
$$
and the limit in the right-hand side is uniform in $(s,x) \in [0, T] \times (\partial \lo \cap K)$. Since $\ve > 0$ is arbitrary and $\eta < \ve$ this implies \eqref{eqn:prop:bd_dj}.
\end{proof}

\begin{cor}\label{cor:limbound}
Under (A1), for any $x\in \lo$ we have
\begin{equation}\label{eqn:cor1}
\limsup_{y\to x, y\in \lo} w(s,y)\leq G\vee H(s,x)
\end{equation}
\end{cor}
\begin{proof}
Notice that
\begin{multline*}
w(s,y)\le
\sup_\tau \ee^y \Big\{ \int_0^{\tau \wedge \tau_\lo} e^{-\alpha u} f(s+u, X(u)) du
 + e^{-\alpha (\tau \wedge \tau_\lo)} G\vee H(s + \tau \wedge \tau_\lo, X(\tau \wedge \tau_\lo)) \Big\}
\end{multline*}
and then continue as in the proof of Proposition \ref{prop:boundary_downjump} replacing $G$ with $G\vee H$.
\end{proof}

The following proposition explores the impact of the value of the functional on the complement of $\lo$ on the value function close to the boundary of $\lo$.
\begin{prop}\label{prop1}
Assume (A1). For each $\ve>0$, $T>0$ and a compact set $K \subseteq E$ there is a compact set $K_\ve \subset \lo$ such that for $x \in K\setminus K_\ve$, $s\in [0,T]$ and $\beta>0$ we have
\begin{equation}\label{eqn:prop1}
w^\beta(s,x) \geq H(s,x)-\ve.
\end{equation}
\end{prop}
\begin{proof}
Fix $\ve' > 0$ and choose $\eta > 0$ such that \eqref{eqn:prop_bound_1}-\eqref{eqn:short_time_cont} hold for the function $H$. By the definition of $w^\beta$ we have
$$
w^\beta(s,x) \ge -\|f\| \eta - \frac{\|f\|}{\alpha} \prob^x \{ \tau_\lo > \eta \} + \ee^{x}\left\{e^{-\alpha \tau_\lo}H(s+\tau_\lo,X(\tau_\lo))\right\}.
$$
Splitting the last term depending on whether $\tau_\lo$ is greater or smaller than $\eta$ and doing analogous estimates as in the proof of Proposition \ref{prop:boundary_downjump} we obtain the following lower bound
$$
w^\beta(s,x) \ge (1 - h_\eta(x)) H(s,x) - h_\eta(x) \big(\frac{\|f\|}{\alpha} + \|H\|\big) - \ve' (2 + 6 \|H\| + \|f\|).
$$
By arbitrariness of $\ve' > 0$ and continuity of $h_\eta$ we can choose $K_\ve$ such that \eqref{eqn:prop1} is satisfied.
\end{proof}
According to Proposition \ref{prop1}, Assumption (A1) guarantees the "migration" of $H$ into $\lo$, i.e., the function $H$ provides a lower bound for $w^\beta$ when $x$ approaches $\partial \lo$. As $w^\beta$ is the lower bound for $w$ (see Proposition \ref{prop:pointwise}) this property is shared by the value function $w$. In particular, when $H \ge G$ the value function smoothly rises to the upper level $H$ on the boundary of $\lo$.

\begin{proof}[Proof of Theorem \ref{thm:limit}]
From \eqref{eqn:prop1}, letting first $\beta \to \infty$ and then $\ve \to 0$ we obtain
$$
\liminf_{y\to x, y \in \lo} w(s,x)\geq  H(s,x).
$$
Since $G$ is continuous and $w(s,y)\geq G(s,y)$ on $\lo$ this extends to
$$
\liminf_{y\to x, y \in \lo} w(s,x)\geq G \vee H(s,x).
$$
Corollary \ref{cor:limbound} and the above inequality imply that the limit in \eqref{eq:limit} exists and equals $G \vee H$.
\end{proof}

\section{Continuity of $w$ and existence of optimal stopping times}\label{sec:cont_w}

Let $\lde$ denote the set of functions $\varphi(s, x)$ admitting the following decomposition:
\begin{equation}\label{eqn:decomp_D}
\varphi(s, x) = \ee^x \Big\{ \int_0^{\tau_\lo} e^{-\alpha u} \varphi_1(s+u, X(u)) du + e^{-\alpha \tau_\lo} \varphi_2(s + \tau_\lo, X(\tau_\lo)) \Big\}
\end{equation}
for $\varphi_1, \varphi_2 \in C_0([0, \infty) \times E)$.

\begin{lemma}\label{lem:density_D}
The set $\lde$ is a dense subset of $\mathcal C_0([0, \infty) \times E)$.
\end{lemma}
\begin{proof}
It follows immediately from the proof of Lemma 4 in \cite{stettner}.
\end{proof}

\begin{lemma}\label{lem:bar_w_beta}
If $G$ has decomposition \eqref{eqn:decomp_D} with $\varphi_1= g_1, \varphi_2=g_2 \in \mathcal C_0([0, \infty) \times E)$ then
$$
w^\beta(s, x) - G(s, x) \ge - \frac{\|f - g_1\|}{\alpha + \beta} - \ee^x \big\{ e^{-(\alpha + \beta) \tau_\lo} \| H - g_2 \| \big\}.
$$
If, moreover, $G \le H$ then
$$
w^\beta(s, x) - G(s, x) \ge - \frac{\|f - g_1\|}{\alpha + \beta}.
$$
\end{lemma}
\begin{proof}
Define $\bar w^\beta(s,x) = w^\beta (s,x) - G (s,x)$. Decomposition \eqref{eqn:decomp_D} of $G$ and representation \eqref{eqn:def_w_beta} of $w^\beta$ imply
\begin{align*}
\bar w^\beta (s, x) &= \ee^x\Big\{\int_0^{\tau_\lo} e^{-\alpha u} \big[ f - g_1
+ \beta (\bar w^\beta )^- \big]\big(s+u,X(u)\big) du +e^{-\alpha \tau_\lo} (H - g_2)\big(s + \tau_\lo, X(\tau_\lo)\big)\Big\}.
\end{align*}
By Lemma \ref{lem:transf} we have the following equivalent form of the above equation
\begin{multline*}
\bar w^\beta (s, x) = \ee^x\Big\{\int_0^{\tau_\lo} e^{-(\alpha+\beta) u} \big[ f - g_1 + \beta (\bar w^\beta)^- +  \beta \bar w^\beta \big]\big(s+u,X(u)\big) du\\
+e^{-(\alpha+\beta) \tau_\lo} (H - g_2)\big(s + \tau_\lo, X(\tau_\lo)\big)\Big\}.
\end{multline*}
Since $(\bar w^\beta)^- + \bar w^\beta \ge 0$, we obtain
\begin{align*}
\bar w^\beta (s, x) &\ge \ee^x\Big\{\int_0^{\tau_\lo} e^{-(\alpha + \beta) u} (f - g_1)\big(s+u,X(u)\big) du+e^{-(\alpha + \beta) \tau_\lo} (H - g_2)\big(s + \tau_\lo, X(\tau_\lo)\big)\Big\}.
\end{align*}
This implies the first statement of the lemma.

Due to decomposition \eqref{eqn:decomp_D} we have $g_2(s,x) = G(s,x)$ for $x \notin \lo$. Together with the condition $G \le H$ this yields $(H - g_2)\big(s + \tau_\lo, X(\tau_\lo)\big) \ge 0$.
\end{proof}

\begin{thm}\label{thm:cont_for_upjump}
Assume (A1) and $G \le H$. The value function $w$ is continuous on $E$ and an optimal stopping moment is given by
\begin{equation}\label{eqn:opt_stop_formula}
\tau^*(s) = \inf \{ t \ge 0: w(s+t,X(t)) \le G(s+t,X(t)) \text{\ \ or\ \ } X(t)\notin \lo \}.
\end{equation}
\end{thm}
\begin{proof}
Functions $w^\beta$ are continuous (by Lemma \ref{lem:solv_w_beta}), increasing in $\beta$ and dominated by $w$. Therefore, it suffices to estimate the difference $w - w^\beta$. For functions $G$ with decomposition \eqref{eqn:decomp_D},  Lemma \ref{lem:bar_w_beta} and equation \eqref{eqn:wbeta_form_tau} give
$$
w^\beta(s,x) \ge w (s,x) - \frac{\|f - g_1\|}{\alpha + \beta}.
$$
Since $\lde$ is dense in $\mathcal C_0([0, \infty) \times E)$ (Lemma \ref{lem:density_D}) we also obtain the continuity of $w$ for $G \in \mathcal C_0([0, \infty) \times E)$.

The extension of this result to continuous bounded $G$ uses Proposition \ref{prop:01} in the appendix. Fix a compact set $K \subseteq E$ and $S \ge 0$. For any $T, \ve > 0$ there is a compact set $L \subseteq E$ such that
$$
\prob^x\big( X(t) \notin L \text{ for some $t \in [0, T]$} \big) < \ve, \qquad x \in K.
$$
Define $\tl G(s,x) = e^{-\rho(x, L) - (s - (S + T))^+} G(s,x)$, where $\rho(x, L)$ denotes the distance of $x$ from the set $L$. Let $\tl w$ be the value function corresponding to $\tl G$. Since $\tl G \in \mathcal C_0([0, \infty) \times E)$, preceding results imply that $\tl w$ is continuous. We also have $\|w(s,x) - \tl w(s,x)\| \le (e^{-\alpha T} + \ve) \big( \|f\| / \alpha + \|G\| + \|H\|)$ for $x \in K$ and $s \in [0, S]$. Since $T$ and $\ve$ are arbitrary this implies continuity of $w$ on $[0, S] \times K$. By the arbitrariness of $S, K$ the value function $w$ is continuous on its whole domain.

Define for $\ve>0$
\begin{equation}\label{eqn: epsiopt_formula}
\tau^\ve(s) = \inf \{ t \ge 0: w(s+t,X(t)) \le G(s+t,X(t))+\ve \text{\ \ or\ \ } X(t)\notin \lo \}
\end{equation}
and
\begin{equation}\label{eqn:betastop_formula}
\tau_\beta(s) = \inf \{ t \ge 0: w^\beta(s+t,X(t)) \le G(s+t,X(t)) \text{\ \ or\ \ } X(t)\notin \lo \}.
\end{equation}
Fix $\delta > 0$ and $T > 0$. By Proposition \ref{prop:01} for a given $x \in E$ there is a compact set $K_\delta$ such that $\prob^x \{A_\delta\} \ge 1-\delta$, where $A_\delta = \{ X(t) \in K_\delta \ \forall t \in [0, T] \}$.
From \eqref{eqn:def_w_beta}, due to the Markov property of $(X(t))$, we obtain
$$
\begin{aligned}
w^\beta (s, x)
&\le \Big[\prob^x \{A_\delta^c\} + e^{-\alpha T} \prob^x \{ A_\delta \text{ and } \tau_\beta(s)\vee \tau^\ve(s) > T \} \Big] \big(\|G\| + \|H\| + \frac{\|f\|}{\alpha}\big) \\
&\mop{8}+ \ee^x\Big\{\ind{A_\delta} \int_0^{\sigma^*_{\ve, \beta,T}(s)} e^{-\alpha u} \big[ f+ \beta (G - w^\beta)^+ \big]\big(s+u,X(u)\big) du\\[3pt]
&\mop{32} +\ind{A_\delta} e^{-\alpha \sigma^*_{\ve, \beta,T}(s)} w^\beta\big(s + \sigma^*_{\ve, \beta,T}(s), X(\sigma^*_{\ve, \beta,T}(s))\big)\Big\},
\end{aligned}
$$
where $\sigma^*_{\ve, \beta,T}(s) = \tau_\beta(s)\wedge \tau^\ve(s)\wedge\tau_\lo\wedge T$.
Notice that by the uniform convergence on compact subsets of $w^\beta$ to $w$ we have $\prob^x\big\{\tau_\beta(s) \wedge T < \tau^\ve(s) \wedge T \text{ and } A_\delta \big\} \to 0$ as $\beta \to \infty$. Therefore letting $\beta \to \infty$  we obtain by Dominated Convergence Theorem
$$
\begin{aligned}
w(s, x) &\le \big(\delta +e^{-\alpha T}\big)  \big(\|G\| + \|H\| + \frac{\|f\|}{\alpha}\big) \\
&\mop{8} + \ee^x\Big\{\ind{A_\delta} \int_0^{\tau^\ve(s)\wedge \tau_\lo\wedge T} e^{-\alpha u}  f\big(s+u,X(u)\big)du\\
&\mop{32} +\ind{A_\delta} e^{-\alpha \tau^\ve(s)\wedge\tau_\lo\wedge T} w\big(s + \tau^\ve(s)\wedge\tau_\lo\wedge T, X(\tau^\ve(s)\wedge\tau_\lo\wedge T)\big)\Big\}.
\end{aligned}
$$
Proposition \ref{prop:01} implies that $(A_\delta)$ form an increasing sequence of subsets when $\delta \to 0$ and $\lim_{\delta \to 0} \prob^x\{A_\delta\} = 1$. Therefore, letting $\delta \to 0$ we get
$$
\begin{aligned}
w(s, x) &\le e^{-\alpha T}\big(\|G\| + \|H\| + \frac{\|f\|}{\alpha}\big) \\
&\mop{8} + \ee^x\Big\{\int_0^{\tau^\ve(s)\wedge \tau_\lo\wedge T} e^{-\alpha u} f\big(s+u,X(u)\big) du\\
&\mop{32} +e^{-\alpha \tau^\ve(s)\wedge\tau_\lo\wedge T} w\big(s + \tau^\ve(s)\wedge\tau_\lo\wedge T, X(\tau^\ve(s)\wedge\tau_\lo\wedge T)\big)\Big\}.
\end{aligned}
$$
Now taking the limit $T \to \infty$ yields
\begin{equation}\label{eqn:Belstopped}
w(s, x) \le \ee^x\Big\{\int_0^{\tau^\ve(s)\wedge \tau_\lo} e^{-\alpha u} f\big(s+u,X(u)\big) du+e^{-\alpha \tau^\ve(s)\wedge\tau_\lo} w\big(s + \tau^\ve(s)\wedge\tau_\lo, X(\tau^\ve(s)\wedge\tau_\lo)\big)\Big\}.
\end{equation}
Note that $\tau^\ve(s) \to \tau^*(s)$, as $\ve \to 0$, and, further, by quasi-leftcontinuity of the process $(X(t))$ (see, e.g., \cite{dynkin}) we also have $X(\tau^\ve(s))\to X(\tau^*(s))$, $\prob^x$-a.s.. Consequently letting $\ve \to 0$ in \eqref{eqn:Belstopped} and using the continuity of $w$ give
$$
\begin{aligned}
w(s, x) &\le \ee^x\Big\{\int_0^{\tau^*(s)\wedge \tau_\lo} e^{-\alpha u} f\big(s+u,X(u)\big) du +e^{-\alpha \tau^*(s)\wedge\tau_\lo} w\big(s + \tau^*(s)\wedge\tau_\lo, X(\tau^*(s)\wedge\tau_\lo)\big)\Big\}.
\end{aligned}
$$
This implies that there is an optimal stopping time dominating $\tau^*(s)$. The optimality of $\tau^*(s)$ is now obvious.
\end{proof}


To prove the continuity of the value function $w$ without the requirement of an upward jump we introduce the following assumptions:
\begin{assumption}
\item[(A2)] \label{ass:cont}
$\lim_{\eta \to 0} \prob^x\{ \tau_\lo < \eta \} = 0$ uniformly in $x$ from compact subsets of $\lo$.
\item[(A3)] \label{ass:strong_feller}
$(X(t))$ is strongly Feller, i.e., the mapping $x \mapsto \ee^x \{ h(X(t)) \}$ is continuous for any measurable bounded function $h$ and $t > 0$.
\end{assumption}

Before we formulate Theorem \ref{thm:cont_w_in _O} we prove three auxiliary results. Lemma \ref{lem:a3'} shows that under \ref{ass:strong_feller} the time-state process semigroup maps time-continuous bounded functions into functions continuous in both parameters. Lemma \ref{lem:rightcont} states that the weak Feller continuity of the process $X(t)$ is sufficient for the continuity of the value function $w$ in the time parameter $s$. Lemma \ref{lem:a2} shows that \ref{ass:cont} follows from \ref{ass:a1}.

\begin{lemma}\label{lem:a3'}
Under assumption \ref{ass:strong_feller}, the mapping
$$
(s,x)\mapsto \ee^{x}\left\{F(s+h,X(h))\right\}
$$
is continuous for $h>0$ and a bounded measurable function $F$, provided that the mapping $s \mapsto F(s,x)$ is continuous uniformly in $x$ in compact subsets of $E$.
\end{lemma}
\begin{proof}
Fix a compact set $K \subseteq E$ and $T, \ve > 0$. By Proposition \ref{prop:01} there is a compact set $L \subseteq E$ such that $\sup_{x \in K} \prob^x \{ X(h) \notin L\} < \ve$. Hence for $(s,x) \in [0,T] \times K$ we have
$$
\big| \ee^x \big\{ F\big(s+h,X(h)\big) \big\} - \Phi(s,x) \big| < \|F\| \ve,
$$
where $\Phi(s,x) = \ee^x \big\{ \ind{X(h) \in L} F\big(s+h,X(h)\big) \big\}$. Let $(s_n, x_n) \to (s,x)$ such that $(s_n, x_n) \in [0,T] \times K$ for all $n$. By the continuity of $F$ in $s$ and by assumption \ref{ass:strong_feller}, for sufficiently large $k$, we have
\begin{multline*}
\lim_{n \to \infty} \big| \Phi(s_n, x_n) - \Phi(s,x) \big|
 \le \lim_{n \to \infty} \big| \Phi(s_k, x_n) - \Phi(s,x_n) \big| + \lim_{n \to \infty} \big| \Phi(s, x_n) - \Phi(s,x) \big| = \ve + 0.
\end{multline*}
By the arbitrariness of $\ve$ this completes the proof.
\end{proof}

\begin{lemma}\label{lem:rightcont}
The mapping $s \mapsto w(s,x)$ is continuous uniformly in $x$ in compact subsets of $E$.
\end{lemma}
\begin{proof}
Assume that $s_n\to s$ and fix a compact set $K\subseteq E$. Since functions $f$, $G$ and $H$ are bounded and the discount rate $\alpha>0$, for any $\ve>0$ there is $T>0$ such that $|J(s_n,x,\tau)-J(s_n,x,\tau \wedge T)|\leq \ve$ for all $x \in K$ and $n=1, 2, \ldots$. By Proposition \ref{prop:01} there is a compact set $L \subseteq E$ such that for all $x \in K$ and $\tau\leq T$ we have
\begin{align*}
&\ee^x\ind{\exists_{t\in [0,T]} X(t)\notin L} \Big\{ \int_0^{\tau \wedge \tau_\lo} e^{-\alpha u} f\big(s+u,X(u)\big) du \\
&\mop{18}+ \ind{\tau<\tau_\lo}e^{-\alpha \tau} G\big(s+\tau,X(\tau)\big) +\ind{\tau\geq \tau_\lo} e^{-\alpha \tau_\lo} H\big(s + \tau_\lo, X(\tau_\lo)\big)\Big\} \le \ve.
\end{align*}
Uniform continuity of the functions $f$, $G$, and $H$ on $[0,T]\times L$ yields
\begin{align*}
\ee^x\ind{\forall_{t\in [0,T]} X(t)\in L} \Big\{&\int_0^{\tau \wedge \tau_\lo} e^{-\alpha u} \left[f\big(s+u,X(u)\big)-f\big(s_n+u,X(u)\big)\right] du \\
&\mop{0}+ \ind{\tau<\tau_\lo}e^{-\alpha \tau} \left(G\big(s+\tau,X(\tau)\big)- G\big(s_n+\tau,X(\tau)\big)\right)\\
&\mop{0}+ \ind{\tau\geq \tau_\lo} e^{-\alpha \tau_\lo} \left(H\big(s + \tau_\lo, X(\tau_\lo)\big)-H\big(s_n + \tau_\lo, X(\tau_\lo)\big)\right)\Big\}
\leq \ve
\end{align*}
for $\tau \le T$, a sufficiently large $n$ and $x \in K$. Consequently $w(s_n,x) \to w(s,x)$ as $n\to \infty$ uniformly in $x \in K$.
\end{proof}

\begin{lemma}\label{lem:a2}
Assumption \ref{ass:a1} implies \ref{ass:cont}.
\end{lemma}
\begin{proof}
By Lemma \ref{lem:cont_stopped_potential} the function $g_{\gamma}(x) = \ee^x \{ e^{-\gamma \tau_\lo}\}$ is continuous on $\lo$ for any $\gamma > 0$. By dominated convergence theorem $g_\gamma(x)$ converges to $0$ when $\gamma \to \infty$ and $x \in \lo$. This convergence is monotone and, due to Dini's theorem, uniform on compact subsets of $\lo$. Chebyshev's theorem yields
$$
\prob^x \{ \tau_\lo < \eta \} = \prob^x \{ e^{-\tau_\lo/\eta} > e^{-1} \} \le e\, g_{1/\eta}(x).
$$
The right-hand side converges to $0$, when $\eta \to 0$, uniformly on compact subsets of $\lo$, which completes the proof.
\end{proof}
\begin{remark}\label{rem:a2'}
Assumption \ref{ass:cont} holds if the process $(X(t))$ satisfies the following continuity condition:
\begin{assumption}
\item[(A2')] \label{ass:a2'}
for any $\ve>0$
$$
\lim_{t \to 0} \prob^x\big\{\sup_{s\in[0,t]}\rho(x,X(s)) \geq \ve \big\} = 0
$$
uniformly in $x$ from compact sets.
 \end{assumption}
 Such assumption is satisfied for a wide variety of Markov processes which are solutions to the stochastic differential equations with Levy noise with bounded coefficients. To prove this we simply use the Doob's maximal inequality to the martingale terms in the stochastic differential equation (see e.g. Theorem 1.3.8(iv) in \cite{KS}).
\end{remark}

Let for $h>0$
\begin{equation}
w_h(s,x)=\ee^{x}\Big\{ \int_0^h e^{-\alpha u} f\big(s+u,X(u)\big) du + e^{-\alpha h}w(s+h,X(h))\Big\}.
\end{equation}

\begin{thm}\label{thm:cont_w_in _O}
Under \ref{ass:cont} and \ref{ass:strong_feller}, the function $w$ is continuous on $[0, \infty) \times \lo$. Assume additionally (A1). The penalized functions $w^\beta$ are continuous and converge to $w$ uniformly on compact subsets of $[0, \infty) \times \lo$. An $\ve$-optimal stopping time is given by
$$
\tau^\ve(s) = \inf \{ t \ge 0: w(s+t,X(t)) \le G(s+t,X(t))+\ve \text{\ \ or\ \ } X(t)\notin \lo \}.
$$
\end{thm}
\begin{proof}
By Lemmas \ref{lem:a3'} and \ref{lem:rightcont} the function $w_h$ is continuous in $(s,x)$. Let $\tau_\lo^h = \inf\{t\geq h: X(t) \notin \lo\}$ and
\begin{multline*}
J^h(s, x, \tau) = \ee^x\Big\{\int_0^{\tau \wedge \tau_\lo^h} e^{-\alpha u} f\big(s+u,X(u)\big) du\\
+ \ind{\tau<\tau_\lo^h}e^{-\alpha \tau} G\big(s+\tau,X(\tau)\big) +\ind{\tau\geq \tau_\lo^h} e^{-\alpha \tau_\lo^h} H\big(s + \tau_\lo, X(\tau_\lo^h)\big)\Big\}.
\end{multline*}
By Theorem 3b of \cite{mertens} applied to the Markov process consisting of a pair $(s+t,X(t))$ we have
$$
w_h(s,x) = \sup_{\tau\geq h}J^h(s,x,\tau).
$$
Consider an auxiliary value function
$$
\tilde{w}_h(s,x)=\sup_{\tau\geq h} J(s, x, \tau).
$$
We have the following inequalities
\begin{equation}\label{oszac1}
|w_h(s,x)-\tilde{w}_h(s,x)|\le C\, \prob^x\big\{\tau_\lo<h \big\}
\end{equation}
and
\begin{equation}\label{oszac2}
0 \le w(s,x)-\tilde{w}_h(s,x) \leq \sup_{\tau} \big\{ J(s, x, \tau)-J(s,x,\tau_h) \big\} =: I_h(s, x),
\end{equation}
where $\tau_h=\tau \vee h$ and $C>0$. Assumption \ref{ass:cont} implies the difference $|w_h - \tilde w_h|$ converges to $0$ as $h \to 0$ uniformly on compact subsets of $[0, \infty) \times \lo$. The proof of uniform convergence of $I_h$ is more involved. First notice
\begin{align*}
I_h(s, x) &= \sup_{\tau \le h} \big\{ J(s, x, \tau)-J(s,x,h) \big\} \\
&\le \|f\| h + 2(\|G\| + \|H\|)\, \prob^x \{\tau_\lo < h \} + \sup_{\tau \le h} \ee^x\{G(s+\tau, X(\tau))\}- \ee^x \{G(s+h, X(h))\}.
\end{align*}
It suffices to prove that as $h \to 0$
\begin{equation}\label{eqn:proof_Ih}
\sup_{\tau \le h} \ee^x\{G(s+\tau, X(\tau))\} - \ee^x \{G(s+h, X(h))\} \to 0
\end{equation}
uniformly in $s,x$ in compact subsets of $[0, \infty) \times \lo$. By Proposition \ref{cor:03} we have
$$
\lim_{h \to 0} \ee^x \{G(s+h, X(h))\} = G(s,x)
$$
uniformly in $s,x$ in compact subsets. Let $v_h(s,x) = \sup_{\tau \le h} \ee^x\{G(s+\tau, X(\tau))\}$. By weak Feller property this function is continuous (see, e.g., \cite[Corollary 3.6]{palczewski2008} or \cite{Z}). By dominated convergence theorem and the right-continuity of trajectories of $X$ we have $\lim_{h \to 0} v_h(s,x) = G(s,x)$. Since this convergence is monotone and functions $v_h$ and $G$ are continuous Dini's theorem implies that $v_h$ tends to $G$ uniformly on compact sets. This completes the proof of \eqref{eqn:proof_Ih}. Consequently, $w_h(s,x)$ converges to $w(s,x)$ as $h\to 0$ uniformly in compact subsets of $(s,x) \in [0, \infty) \times \lo$ and $w$ is continuous on $[0, \infty) \times \lo$.

Assume (A1). By Corollary \ref{cor:cont_wbeta} functions $w^\beta$ are continuous. Dini's Theorem and Proposition \ref{prop:pointwise} imply their uniform convergence to $w$ on compact sets. In an identical way as in Theorem \ref{thm:cont_for_upjump} we prove that $\tau^\ve(s)$ is well-defined and $\ve$-optimal (this last assertion follows directly from \eqref{eqn:Belstopped}).
\end{proof}

Theorem \ref{thm:cont_w_in _O} states the continuity of $w$ in $[0,\infty) \times \lo$. It is also clear that $w$ is continuous on $[0, \infty) \times \lo^c$ because on this set $w$ coincides with $H$. However, if there is a downward jump on the boundary of $\lo$ ($G(s,x) > H(s,x)$ for some $x \in \partial \lo$) the function $w$ has a discontinuity in this point. This follows from the observation that $w \ge G$ on the set $[0, \infty) \times \lo$ and $w = H$ on $[0, \infty) \times \lo^c$. Therefore, the statement of the above theorem cannot be strengthened. This also implies that an optimal stopping time might not exist as the following example shows.

\begin{example}
Let $E = \er$ and $X(t)$ be a Brownian motion. Take $\lo = (-\infty, 1)$ and $\alpha < 1/2$. It is easy to see that assumptions \ref{ass:a1}-\ref{ass:strong_feller} are satisfied. Put $G(s,x) = \min(e^x, e)$ and $H(s,x) = f(s,x) = 0$. Notice that these functions do not depend on $s$, which implies that the value function is also time-independent. We shall, therefore, skip $s$ in the notation.
\begin{lemma}\label{lem:example}
In the setting of the example,
\begin{enumerate}
\item for $x < 1$ and $t \ge 0$ we have
\begin{equation}\label{eqn:computed_example}
l(t, x) := \ee^x\big\{e^{-\alpha t} \ind{X(t) < 1} e^{X(t)} \big\} = e^{(\frac12 - \alpha) t + x} \Phi \Big(\frac{1-x-t}{\sqrt{t}}\Big),
\end{equation}
where $\Phi$ is the standard normal cumulative distribution function,
\item $\displaystyle w(x) \ge l(t, x)$, for $x < 1$ and $t \ge 0$,
\item $w(x) > G(x)$, for $x < 1$.
\end{enumerate}
\end{lemma}
\begin{proof}[Sketch of the proof]
The formula \eqref{eqn:computed_example} can be calculated directly using the normality of $X(t)$. To prove (2), define a sequence of stopping times $\tau_n = \inf \{t: X(t) \ge 1 - 1/n\}$. Clearly,
$$
w(x) \ge \ee^x \big\{ e^{-\alpha (\tau_n \wedge t)} e^{X(\tau_n \wedge t)} \big\}
$$
and
$$
\lim_{n \to \infty} \ee^x \big\{ e^{-\alpha (\tau_n \wedge t)} e^{X(\tau_n \wedge t)} \big\} \ge l(t, x).
$$
The proof of the last assertion rests on the observation that $G(x) = l(0, x)$ and $\frac{\partial}{\partial t} l(0,x) > 0$ for $x < 1$.
\end{proof}

Assume that there exists an optimal stopping moment $\tau^*$ for some $x^* < 1$, i.e., $w(x^*) = \ee^{x^*} \{ e^{-\alpha \tau^*} \ind{\tau^* < \tau_\lo} G(X(\tau^*)) \}$. From the strong Markov property of the process $X(t)$ we infer that $G(X(\tau^*)) = w(X(\tau^*))$, $\prob^{x^*}$-a.s. on $\{\tau^* < \tau_\lo\}$. Since $w(x^*) \ge e^{x^*}$ we have $\prob^{x^*} (\tau^* < \tau_\lo) > 0$. This is a contradiction with assertion (3) of Lemma \ref{lem:example}.
\end{example}

\begin{remark}
Penalty method offers a numerical procedure for solution of optimal stopping problems. Lemma \ref{lem:wbeta_form_tau} provides an estimate of the error: $\|w - w^\beta\| \le \|(G - w^\beta)^+\|$. This error decreases as $\beta$ increases: by Proposition \ref{prop:pointwise} $w^\beta$ forms a non-decreasing sequence of functions converging to $w$. Under \ref{ass:a1} functions $w^\beta$ are continuous (c.f. Corollary \ref{cor:cont_wbeta}). Theorems \ref{thm:cont_for_upjump} and \ref{thm:cont_w_in _O} state assumptions under which $w$ is continuous and is approximated by $w^\beta$ uniformly on compact sets.  The continuity of $w$ and $w_\beta$ implies that state space discretization methods can be safely applied. Following Lemma \ref{lem:solv_w_beta} function $w^\beta$ can be computed as a fixed point of a contraction operator $\mathcal T$ given by
\begin{multline*}
\mathcal T \phi(s,x) = \ee^x\Big\{\int_0^{\tau_\lo} e^{-(\alpha+\beta) u} \Big[ f\big(s+u,X(u)\big) + \beta \Big( G\big(s+u,X(u)\big) - \phi\big(s+u,X(u)\big) \Big)^+\\
+\beta \phi\big(s+u,X(u)\big) \Big] du +e^{-(\alpha+\beta) \tau_\lo} H\big(s + \tau_\lo, X(\tau_\lo)\big)\Big\}
\end{multline*}
for a bounded measurable function $\phi$. This operator can be implemented via PDE or Kushner-Dupuis space-time discretization approach (see \cite{KD}). The fixed point is approximated by an iterative procedure with an exponential decrease of the error (due to the contraction property of $\mathcal T$).
\end{remark}

\section{Sufficient conditions for (A1)}\label{sec:conditions_for_A1}
Define for $\eta > 0$
$$
h_{\eta}(x) = \prob^x \{ \tau_\lo > \eta\}.
$$
Consider the following assumption:
\begin{assumption}
\item[(A4)] \label{ass:h_boundary}
$$ \displaystyle \lim_{x \to \partial \lo,\ x \in \lo} h_\eta(x) = 0.$$
\end{assumption}
This assumption ensures that when approaching the boundary of $\lo$ the probability of crossing it in a short time converges to $1$. It is clearly satisfied (by Chebyshev inequality) whenever the mapping $x \to \ee^x \{\tau_\lo\}$ is continuous.  It can be viewed as a complementary assumption to \ref{ass:cont}. We will show that \ref{ass:cont}-\ref{ass:h_boundary} imply (A1) and (A1) is sufficient for \ref{ass:h_boundary}.

\begin{lemma}\label{lem:cont_h_eta_a2_a4}
The function $h_\eta$ is continuous on $E$ under assumptions \ref{ass:cont}-\ref{ass:h_boundary}.
\end{lemma}
\begin{proof}
For $\delta \in (0, \eta)$ define $r_\delta(x) = \ee^x \{ h_{\eta - \delta}(X(\delta))\}$. This function is continuous by \ref{ass:strong_feller}. The difference between $r_\delta$ and $h_\eta$ can be bounded in the following way:
$$
0 \le r_{\delta}(x) - h_\eta(x) \le \prob^x \{ \tau_\lo < \delta \}.
$$
Assumption \ref{ass:cont} states that the right-hand side of the above inequality converges to $0$ as $\delta \to 0$ uniformly in $x$ from compact subsets of $\lo$. Hence, $h_\eta$ is continuous in $\lo$. It is identically zero on $E \setminus \lo$. These two pieces fit continuously at the boundary of $\lo$ because, due to \ref{ass:h_boundary}, $h_\eta(x)$ converges to $0$ as $x$ approaches the boundary of $\lo$.
\end{proof}

The continuity of $h_\eta$ implies uniformity of the limit in assumption \ref{ass:h_boundary}, which is formalized in the following corollary.
\begin{cor}\label{cor:uniform_conv_h_eta}
If $h_\eta$ is continuous then for any compact set $L \subseteq E$ and constants $\eta, \ve > 0$ there is an open set $L_{\eta, \ve}$ such that $ \overline L_{\eta, \ve} \subset \lo$ and
for each $x \in L \setminus L_{\eta, \ve}$ we have $\prob^x (\tau_\lo > \eta) \le \ve$.
\end{cor}

\begin{prop}\label{prop:2.2}
Under \ref{ass:cont}-\ref{ass:strong_feller} the mapping $x \mapsto P^{\tau_\lo}_t h(x)$ is continuous on $E \setminus \partial \lo$ for any bounded measurable function $h$ and $t>0$. If additionally \ref{ass:h_boundary} holds then $P^{\tau_\lo}_t$ maps the space of bounded measurable functions into the space of continuous bounded functions and as a result condition (A1) is satisfied.
\end{prop}
\begin{proof}
Let $h$ be a bounded measurable function. By the strong Feller property \ref{ass:strong_feller}, for $s<t$,  the mapping
$x \mapsto \ee^x\left\{\ee^{X(s)}\left\{\ind{t-s <\tau_\lo} h(X(t-s))\right\}\right\}$ is continuous. Furthermore,
\begin{equation}\label{est1}
\begin{aligned}
\ee^x \{ \ind{t < \tau_\lo} h(X(t))\} &=\ee^x\left\{ \ind{s <\tau_\lo}\ee^{X(s)}\left\{\ind{t-s <\tau_\lo} h(X(t-s))\right\}\right\} \\
& =\ee^x\left\{\ee^{X(s)}\left\{\ind{t-s <\tau_\lo} h(X(t-s))\right\}\right\}- \ee^x\left\{\ind{\tau_\lo \leq s} \ee^{X(s)}\left\{\ind{t-s <\tau_\lo} h(X(t-s))\right\}\right\}.
\end{aligned}
\end{equation}
Therefore
$$
\Big|\ee^x \{ \ind{t < \tau_\lo} h(X(t))\}-\ee^x\left\{\ee^{X(s)}\left\{\ind{t-s <\tau_\lo} h(X(t-s))\right\}\right\}\Big|
\le \|h\|\ \prob^x\{\tau_\lo \le s\} \le \|h\|\ \prob^x\{\tau_\lo < 2s\} \to 0
$$
uniformly on compact subsets of $\lo$ as $s\to 0$ by \ref{ass:cont}. This shows the continuity of $x \mapsto P^{\tau_\lo}_t h(x)$ for $x\in \lo$.
For $x$ in $E\setminus \lo$ we clearly have $\ee^x \{ \ind{t < \tau_\lo} h(X(t))\}=0$. Now we prove continuous fit at the boundary of $\lo$.  Assumption \ref{ass:h_boundary} implies that $|\ee^x \{ \ind{t < \tau_\lo} h(X(t))\}|\leq \prob^x\left\{ t < \tau_\lo\right\} \|h\|$ decreases to $0$ as $x$ approaches the boundary. Hence, $P^{\tau_\lo}_t h$ is continuous on $E$.
\end{proof}

Proposition \ref{prop:2.2} states that \ref{ass:cont}-\ref{ass:h_boundary} are sufficient for Assumption (A1). The following lemma shows that \ref{ass:a1} implies \ref{ass:h_boundary}. Recall that \ref{ass:a1} also implies \ref{ass:cont}, see Lemma \ref{lem:a2}.

\begin{lemma}\label{lem:a1_to_a5}
Under (A1) the function $h_\eta$ is continuous on $E$, which, in particular, implies \ref{ass:h_boundary}.
\end{lemma}
\begin{proof}
Follows from the identity $h_\eta = P^{\tau_\lo}_\eta \mathbf{1}$, where $\mathbf{1}$ denotes a function identically equal $1$.
\end{proof}

\section{Stopping with discontinuities on $\lo^c$}\label{sec:general_func}
In this section we explore a stopping problem with a more general payoff function $F$:
$$
J(s, x, \tau) = \ee^x\Big\{\int_0^{\tau \wedge \tau_\lo} e^{-\alpha u} f\big(s+u,X(u)\big) du
+ e^{-\alpha (\tau \wedge \tau_\lo)} F\big(s + (\tau \wedge \tau_\lo), X(\tau \wedge \tau_\lo)\big)\Big\},
$$
where $f, F$ are measurable bounded functions that are continuous in $s$ uniformly in $x$ from compact sets and $F$ is continuous on $[0, \infty) \times \lo$. In particular, $F$ can be of the form
$$
F(s,x) = \ind{x \in \bar \lo} G(s, x) + \ind{x \notin \bar \lo} H(s,x),
$$
where $G,H$ are continuous bounded functions. This is a complementary problem to the one described in preceding sections: a discontinuity of the payoff manifests itself only when the process $(X(t))$ jumps to $\bar \lo^c$ at the time $\tau_\lo$. For a continuous process $(X(t))$ the form of $F$ outside of $\bar \lo$ is irrelevant and the problem simplifies to stopping with a continuous payoff function $G$. However, if $(X(t))$ jumps at $\tau_\lo$, the process migrates to the set $\bar \lo^c$ and the value of the functional is given by $H$.

Define a value function $w(s,x) = \sup_{\tau} J(s,x,\tau)$.

\begin{prop}
Under \ref{ass:cont} and \ref{ass:strong_feller}, the function $w$ is continuous in $\lo$.
\end{prop}
\begin{proof}
As in Lemma \ref{lem:rightcont}, using continuity of $s \mapsto \big(f(s,x), F(s,x)\big)$ uniform in $x$ from compact sets we obtain that $s \mapsto w(s,x)$ is continuous uniformly in $x$ from compact sets. The rest of the proof follows similar lines as the proof of Theorem \ref{thm:cont_w_in _O}.
\end{proof}

Define a penalized equation (c.f. equation \eqref{eqn:def_w_beta}):
$$
w^\beta (s, x) = \ee^x\Big\{\int_0^{\tau_\lo} e^{-\alpha u} \big[ f  + \beta (F - w^\beta)^+ \big]\big(s+u,X(u)\big) du
+e^{-\alpha \tau_\lo} F\big(s + \tau_\lo, X(\tau_\lo)\big)\Big\}.
$$
As previously, this function is a fixed point of a contraction operator (see the proof of Lemma \ref{lem:solv_w_beta}).
To establish the convergence of $w^\beta$ to $w$, we need the following technical lemma:
\begin{lemma}\label{lem:ass_cont_F}
Under \ref{ass:a1} and \ref{ass:strong_feller} the mapping
$$
(s,x) \mapsto \ee^{x}\left\{F(s+\tau_\lo,X(\tau_\lo))\right\}
$$
is continuous in $[0, \infty) \times \lo$ for any bounded measurable function $F$ that is continuous in $s$ uniformly in $x$ from compact sets.
\end{lemma}
\begin{proof}
Lemma \ref{lem:a3'} implies that $\Phi_h(s,x) = \ee^{x}\big\{\ee^{X(h)}\big\{F(s+h+\tau_\lo,X(\tau_\lo))\big\} \big\}$ is continuous for $h > 0$. Note that
\begin{align*}
&\ee^{x}\left\{F(s+\tau_\lo,X(\tau_\lo))\right\}\\
&=\ee^{x}\Big\{\ee^{X(h)}\left\{F(s+h+\tau_\lo,X(\tau_\lo))\right\} + \ind{\tau_\lo < h}\left(F(s+\tau_\lo,X(\tau_\lo))-\ee^{X(h)}\left\{F(s+h+\tau_\lo,X(\tau_\lo))\right\}\right)\Big\}.
\end{align*}
Hence
$$
\big| \ee^{x}\left\{F(s+\tau_\lo,X(\tau_\lo))\right\} - \Phi_h(s,x)\big|
\le 2\|F\| \prob^x\{\tau_\lo < h\}.
$$
The right-hand side converges to $0$ uniformly in $x$ from compact subsets of $\lo$, as $h\to 0$, by virtue of Lemma \ref{lem:a2}.
\end{proof}

\begin{prop}\label{prop:general_F_prop2}
Under \ref{ass:a1} and \ref{ass:strong_feller}:
\begin{enumerate}
\item There is a unique measurable bounded solution $w^{\beta}$ to the above penalized equation. This function is continuous on $[0, \infty) \times \lo$.
\item Functions $w^{\beta}$ are non-decreasing in $\beta$.
\end{enumerate}
\end{prop}
\begin{proof}
Existence of a unique bounded measurable solution follows from Lemma \ref{lem:solv_w_beta}. Lemma \ref{lem:a3'} and \ref{lem:ass_cont_F} imply the continuity of $w^\beta$. Assertion (2) follows from Lemma \ref{lem:wbeta_form_mbeta}.
\end{proof}

\begin{prop}
Under \ref{ass:a1} and \ref{ass:strong_feller} the sequence of functions $w^\beta$ converges to $w$ and this convergence is uniform on compact subsets of $[0, \infty) \times \lo$.
\end{prop}
\begin{proof}
Using continuity of $F$ on $[0, \infty) \times \lo$, in a similar way as in Lemma \ref{lem:wbeta_form_tau} we obtain
$$
w^\beta (s, x) = \sup_\tau \Big\{ J(s, x, \tau) - \ee^x \big\{\ind{\tau < \tau_\lo} e^{-\alpha \tau}
\big( F - w^\beta \big)^+\big(s+\tau,X(\tau)\big) \big\}\Big\}.
$$
Proceeding as in Proposition \ref{prop:pointwise} we prove the pointwise convergence of $w^\beta$ to $w$. By Proposition \ref{prop:general_F_prop2}, functions $w^\beta$ are non-decreasing in $\beta$, which implies, by Dini's theorem, uniform convergence on compact sets of $[0, \infty) \times \lo$.
\end{proof}

\section{Infinite time horizon}\label{sec:infinite}
Consider an optimal stopping problem with infinite horizon
\begin{equation}\label{eqn:infinite_stopping}
w^\infty(s,x) = \sup_\tau J^\infty(s,x,\tau),
\end{equation}
where
\begin{equation}\label{eqn:inffunct}
J^\infty(s,x,\tau)=\ee^x \Big\{ \int_0^\tau e^{-\alpha u} f\big(s + u, X(u)\big) du + e^{-\alpha \tau} F\big(s + \tau, X(\tau)\big) \Big\}.
\end{equation}
Assume the process $X(t)$ satisfies the strong Feller property \ref{ass:strong_feller}, $\alpha > 0$ and functions $f, F$ are measurable bounded
and continuous in $s$ uniformly in $x$ from compact sets.

The penalized equation has the following form: for $\beta \ge 0$
\begin{equation}\label{eqn:inf_penalized}
w^{\beta,\infty}(s,x) = \ee^x \Big\{ \int_0^\infty e^{-\alpha u} \big[f + \beta (F - w^{\beta,\infty})^+\big]\big(s + u, X(u)\big)du \Big\}.
\end{equation}

\begin{lemma}\label{lem:inf_wbeta}
Assume \ref{ass:strong_feller}.
\begin{enumerate}
 \item There is a unique (in the space of measurable bounded functions) solution $w^{\beta, \infty}$ of the penalized equation \eqref{eqn:inf_penalized} and this solution is continuous.
\item The mapping $s \mapsto w^\infty(s,x)$ is continuous uniformly in $x$ from compact sets.
\item We have the following equivalent representation of $w^{\beta, \infty}$:
\begin{equation}\label{eqn:wbetainfty_form_mbeta}
\begin{aligned}
w^{\beta,\infty}(s,x)
&= \sup_{b\in M_\beta} \ee^x\Big\{\int_0^{\infty}e^{-\alpha u - \int_0^{u}b(t)dt} \Big[f\big(s+u,X(u)\big)+ b(u)F\big(s+u,X(u)\big)\Big]du \Big\},
\end{aligned}
\end{equation}
where $M_\beta$ is the class of progressively measurable processes with values in $[0,\beta]$.
\end{enumerate}
\end{lemma}
\begin{proof}
(1) Similarly as in Lemma \ref{lem:solv_w_beta} the function $w^{\beta, \infty}$ is a fixed point of the operator
$$
\mathcal T^\infty \phi (s,x) = \ee^x \Big\{ \int_0^\infty e^{-(\alpha +\beta) u} \big(f + \beta \phi + \beta (F - \phi)^+ \big)\big(s + u, X(u)\big) du \Big\}.
$$
This operator is a contraction on the space of measurable bounded functions, which implies that $w^{\beta,\infty}$ is a unique fixed point of $\mathcal T^\infty$ on this space. Lemma \ref{lem:a3'} implies that $\mathcal T^\infty$ maps the space of measurable bounded functions into the space of continuous bounded functions. Hence, $w^{\beta, \infty}$ is continuous.

(2) The proof is similar to that of Lemma \ref{lem:rightcont}. We use the continuity of $s \mapsto \big(f(s,x), F(s,x)\big)$ uniform in $x$ from compact sets.

(3) This assertion follows immediately from Lemma \ref{lem:wbeta_form_mbeta}   with $\lo=E$.
\end{proof}

\begin{lemma}\label{lem:inf_difference_estimate}
Assume there is $\ga \subseteq E$ such that for $x \in \ga$
$$
F(s,x) = R_\alpha \phi(s,x) := \ee^x \Big\{ \int_0^\infty e^{-\alpha u} \phi(s+u, X(u)) du \Big\},
$$
where $\phi: [0, \infty) \times \ga \to \er$ is measurable and bounded.
Then
$$
(F - w^{\beta,\infty})^+(s,x) \le \|f - \phi\|\ \ee^x \Big\{ \int_0^\infty e^{-\alpha u - \beta \int_0^u \ind{X(t) \in \ga} dt}du \Big\},\qquad
(s,x) \in [0, \infty) \times \ga.
$$
\end{lemma}
\begin{proof}
First notice that for any bounded measurable function $\zeta$ the following representations are equivalent
\begin{equation}\label{eqn:representation_b}
\begin{aligned}
v(s,x) &= \ee^x \Big\{ \int_0^\infty e^{-\alpha u} \zeta\big(s+u, X(u)\big) du \Big\},\\
v(s,x) &= \ee^x \Big\{ \int_0^\infty e^{-\alpha u - \int_0^u b(t) dt} \Big[\zeta\big(s+u, X(u)\big) + b(u) v\big(s+u, X(u)\big) \Big] du \Big\}
\end{aligned}
\end{equation}
for any bounded progressively measurable process $b(t)$ (compare to Lemma \ref{lem:transf} with $\lo=E$).

Define
\begin{equation}\label{eqn:integral_repres}
\hat w^{\beta, \infty}(s,x) = \ee^x \Big\{ \int_0^\infty e^{-\alpha u} \big[f + \beta (\bar w^{\beta, \infty})^- - \phi\big]\big(s + u, X(u)\big) du \Big\},
\end{equation}
where $\bar w^{\beta, \infty} = w^{\beta, \infty} - F$. Notice that $\hat w^{\beta, \infty}$ coincides with $\bar w^{\beta, \infty}$ on $[0, \infty) \times \ga$.
Applying equivalence \eqref{eqn:representation_b} for $\zeta = f + \beta (\bar w^{\beta, \infty})^- - \phi$ and $b(u) = \beta \ind{X(u) \in \ga}$ yields
\begin{align*}
&\hat w^{\beta, \infty}(s,x) \\
&= \ee^x \Big\{ \int_0^\infty e^{-\alpha u- \beta \int_0^u \ind{X(t) \in \ga} dt} \Big( \big[f + \beta (\bar w^{\beta, \infty})^- - \phi\big]\big(s + u, X(u)\big)+ \beta \ind{X(u) \in \ga} \hat w^{\beta, \infty}\big(s + u, X(u)\big) \Big)du \Big\}.
\end{align*}
Since $(\bar w^{\beta, \infty})^- + \hat w^{\beta, \infty} \ge 0$ on $[0, \infty) \times \ga$ we have
\begin{align*}
\hat w^{\beta, \infty}(s,x) &\ge \ee^x \Big\{ \int_0^\infty e^{-\alpha u -\beta \int_0^u \ind{X(t) \in \ga} dt} \big[f - \phi\big]\big(s + u, X(u)\big) du \Big\}\\
&\ge - \|f - \phi\| \ \ee^x \Big\{ \int_0^\infty e^{-\alpha u- \beta \int_0^u \ind{X(t) \in \ga} dt} du \Big\}.
\end{align*}
This completes the proof since $\hat w^{\beta, \infty} = w^{\beta, \infty} - F$ on $[0, \infty) \times \ga$.
\end{proof}

We impose the following assumptions on $F$: $F(s,x)=G(s,x)$ for $x\in \lo$ and $F(s,x)=H(s,x)$ for $x\in \lo^c \setminus \partial \lo$, where $G$ and $H$  are bounded continuous functions. Notice that $F$ can be arbitrary on $[0, \infty) \times \partial \lo$ as long as it is continuous in $s$ uniformly in $x$ from compact sets. In particular, $F$ can be equal to $G$ or $H$ on $\partial \lo$.
\begin{lemma} \label{lem:wbetainfty_form_tau}
Under assumption (A3) we have
\begin{equation}\label{eqn:wbetainfty ineq}
\begin{aligned}
w^{\beta, \infty} (s, x) \geq \sup_\tau \Big\{ J^\infty(s, x, \tau) - \ee^x \big\{ e^{-\alpha \tau}
\big( F - w^{\beta, \infty} \big)^+\big(s+\tau, X(\tau)\big) \big\}\Big\},
\end{aligned}
\end{equation}
and if $F \ge G\vee H$ on $\partial \lo$, i.e., $F$ is upper semicontinuous, then $w^{\beta,\infty}$ has the following equivalent representation:
\begin{equation}\label{eqn:wbetainfty_form_tau}
\begin{aligned}
w^{\beta, \infty} (s, x) = \sup_\tau \Big\{ J^\infty(s, x, \tau) - \ee^x \big\{ e^{-\alpha \tau}
\big( F - w^{\beta, \infty} \big)^+\big(s+\tau, X(\tau)\big) \big\}\Big\}.
\end{aligned}
\end{equation}
\end{lemma}
\begin{proof}
We follow the proof of Lemma \ref{lem:wbeta_form_tau}.
For any stopping time $\tau$ we have
\begin{equation*}
w^{\beta, \infty}(s,x) = \ee^x \Big\{ \int_0^\tau e^{-\alpha u} \big( f + \beta(F - w^{\beta,\infty})^+\big)(s + u, X(u)) du + e^{-\alpha \tau} w^{\beta, \infty}(s+\tau, X(\tau) \Big\}.
\end{equation*}
Since $w^{\beta, \infty} \ge F - (F - w^{\beta, \infty})^+$  we obtain \eqref{eqn:wbetainfty ineq}.
Let $\sigma = \inf\{ u: w^{\beta, \infty} (s+u, X(u)) \le F(s+u, X(u)) \}$. On the set $\{\sigma < \infty\}$ the upper semicontinuity of $F$ and the continuity of $w^{\beta, \infty}$ implies $w^{\beta, \infty} (s+\sigma, X(\sigma)) \le F(s+\sigma, X(\sigma))$. Combining this with a trivial result on the set $\{\sigma = \infty\}$ yields
$$
w^{\beta, \infty}(s,x) = J(s, x, \sigma), \qquad \ee^x \big\{ e^{-\alpha \sigma}
\big( F - w^{\beta, \infty} \big)^+\big(s+\sigma, X(\sigma)\big) \big\} = 0.
$$
This, together with \eqref{eqn:wbetainfty ineq}, implies representation \eqref{eqn:wbetainfty_form_tau}.
\end{proof}

In what follows we shall need the following two assumptions:
\begin{assumption}
\item[(A5)] \label{ass:sigma_ve} For any $x \in \partial \lo$ we have
$$
\lim_{\ve \to 0} \sigma_\ve = 0 \text{ $\prob^x$-a.s.,}
\qquad \text{and} \qquad
\lim_{\ve \to 0} \sigma^c_\ve = 0 \text{ $\prob^x$-a.s.,}
$$
where
\begin{align*}
&\sigma_\ve = \inf \{ u \ge 0: \ X(u) \in E \setminus (\lo \cup \Gamma_\ve) \},\\
&\sigma^c_\ve = \inf \{ u \ge 0: \ X(u) \in E \setminus (\lo^c \cup \Gamma_\ve) \},
\end{align*}
and $\Gamma_\ve$ is the $\ve$-neighbourhood of $\partial \lo$:
$$
\Gamma_\ve = \{ x \in E:\ \inf_{y \in \partial \lo} \|x - y\| < \ve \}.
$$
\end{assumption}
\begin{assumption}
\item[(A6)] \label{ass:zero_at_boundary} $\prob^x \{ X(T) \in \partial \lo\} = 0$ for any $x \in E$ and $T> 0$.
\end{assumption}
\begin{remark}
Assumption \ref{ass:sigma_ve} is satisfied whenever each point of $\partial \lo$ is regular for $\lo$ and $E\setminus (\lo\cup \partial \lo)$ (see, e.g., Blumenthal and Getoor \cite{blum} for a definition and properties of regular points). Indeed, \cite[Proposition 10.4]{blum} implies that $T_\lo \ge \lim_{\ve \to 0} \sigma^c_\ve$, where $T_\lo$ is the first hitting time of $\lo$, i.e.
$$
T_\lo = \inf\{ t > 0:\ X(t) \in \lo \}.
$$
Take $x \in \partial \lo$. Its regularity means that $T_\lo = 0$ $\prob^x$-a.s.. Therefore, $\lim_{\ve \to 0} \sigma^c_\ve = 0$ $\prob^x$-a.s.. The convergence of $\sigma_\ve$ to $0$ can be proved in an analogous way.

Assumption \ref{ass:zero_at_boundary} is satisfied whenever Markov process $(X(t))$ has a density at time $T$ with respect to a measure which puts zero weight on the set $\partial \lo$.
\end{remark}

\begin{thm}\label{thm:inf_cont_and_convergence}
Assume \ref{ass:a2'} and \ref{ass:strong_feller}.
\begin{enumerate}
\item $w^\infty$ is continuous on  $[0, \infty) \times (E \setminus \partial \lo)$.
\item $w^{\infty,\infty}(s,x) := \lim_{\beta \to \infty}w^{\beta,\infty}(s,x)$ is lower semicontinuous (l.s.c.) with values in $\er \cup \{\infty\}$.
\item if $F$ is l.s.c., then $w^\infty$ is l.s.c. and $w^{\infty,\infty}\ge w^\infty$.
\item if assumptions \ref{ass:sigma_ve}-\ref{ass:zero_at_boundary} are satisfied and $F \le G \vee H$ on $[0, \infty) \times \partial \lo$, then $w^{\beta,\infty}$ converges to $w^\infty$, as $\beta \to \infty$, uniformly on compact subsets of $[0, \infty) \times (E \setminus \partial \lo)$.
\end{enumerate}
\end{thm}
\begin{proof}
Let $\tilde{w}_h^\infty(s,x)=\sup_{\tau\geq h} J^\infty(s,x,\tau)$. Theorem 3b of \cite{mertens} implies
$$
\tilde{w}_h^\infty(s,x)=\ee^x \Big\{e^{-\alpha h} w^\infty(s+h,X(h))+ \int_0^{h} e^{-\alpha u} f\big(s+u,X(u)\big) du\Big\}.
$$
By Lemmas \ref{lem:a3'} and \ref{lem:inf_wbeta} the function $\tilde{w}_h^\infty$ is continuous for each $h>0$.
Under \ref{ass:cont}, which follows by Remark \ref{rem:a2'} from \ref{ass:a2'},  in the same way as in Theorem \ref{thm:cont_w_in _O} we prove that $\tilde{w}_h^\infty \to w^\infty$, as $h \to 0$, uniformly on compact subsets of $[0, \infty) \times (E \setminus \partial \lo)$. This implies assertion (1).

Assertion (2) follows from Lemma \ref{lem:inf_wbeta}. Indeed, $w^{\beta, \infty}$ is non-decreasing in $\beta$ and continuous for each $\beta$. Hence, the limit $w^{\infty, \infty}$ is well defined and lower semicontinuous.

Define $\tilde{w}^\infty(s,x)=\sup_{\tau > 0} J^\infty(s,x,\tau)$. Notice that $\tilde{w}_h^\infty(s,x)\uparrow \tilde{w}^\infty(s,x)$ as $h\to 0$, which implies $\tilde w^\infty$ is l.s.c. \cite[Theorem 3b]{mertens} implies $w^\infty(s,x)=\max\left\{\tilde{w}^\infty(s,x),F(s,x)\right\}$. Hence, if
$F$ is l.s.c., then the mapping $w^\infty(s,x)$ is l.s.c. as maximum of two l.s.c. functions. Applying \eqref{eqn:wbetainfty_form_mbeta} with  $b(u)=\beta \ind{u\leq h}$ and a sufficiently small $h$ yields $w^{\infty,\infty}\geq F$. Letting $\beta \to \infty$ in \eqref{eqn:wbetainfty ineq} and using  $w^{\infty,\infty}\geq F$ we obtain $w^{\infty,\infty}\geq w^\infty$. This completes the proof of assertion (3).

Last assertion is the most demanding. We assume first that $F = G \vee H$ on $\partial \lo$. We will relax this assumption later. By Lemma \ref{lem:wbetainfty_form_tau} we obtain $w^{\infty, \infty} \le w^\infty$. The proof of the opposite inequality is divided into several steps.
Assertion (4) will then follow from Dini's theorem.

\textbf{Step 1.} Assume $G = R_\alpha g$ and $H = R_\alpha h$, where the functions $g,h:[0, \infty) \times E \to \er$ are continuous bounded  and the resolvent $R_\alpha$ is defined in Lemma \ref{lem:inf_difference_estimate}. It is sufficent to consider $g,h \in \mathcal{C}_0([0, \infty) \times E)$, but it does not simplify the reasoning in any way.

Lemmas \ref{lem:inf_difference_estimate} and \ref{lem:wbetainfty_form_tau} imply the following bound:
\begin{align*}
w^{\beta, \infty} (s,x) \ge \sup_{\tau} \Big[ J^\infty(s,x,\tau)
&- \ee^x \Big\{ \ind{X(\tau) \in \lo} \|f-g\| \phi^\beta_\lo\big(X(\tau)\big) \Big\}\\
&- \ee^x \Big\{ \ind{X(\tau) \in \lo^c \setminus \partial \lo} \|f-h\| \phi^\beta_{\lo^c \setminus \partial \lo} \big(X(\tau)\big) \Big\}\\
&- \ee^x \Big\{\ind{X(\tau) \in \partial \lo} \frac{\|f-g\| \vee \|f - h\|}{\alpha} \Big\}
\Big],
\end{align*}
where, for an open set $A \subset E$, we define
$$
\phi^\beta_{A}(x) = \ee^x \Big\{ \int_0^\infty e^{-\alpha u - \beta \int_0^u \ind{X(t) \in A} dt}du \Big\}.
$$
By dominated convergence theorem, $\lim_{\beta \to \infty} \phi^\beta_A(x) = 0$ for $x \in A$.
Taking a limit as $\beta \to \infty$ yields
\begin{equation}\label{eqn:M_and_boundary}
w^{\infty, \infty} \ge \sup_\tau \Big[ J^\infty(s,x,\tau) - M \prob^x\{X(\tau) \in \partial \lo\} \Big],
\end{equation}
where $M = (\|f - g\| \vee \|f - h\|)/\alpha$.

\textbf{Step 2.} We will show that the supremum in \eqref{eqn:M_and_boundary} can be restricted to stopping times satisfying $\prob^x \{X(\tau) \in \partial \lo\} = 0$. Fix a stopping time $\tau$ and define for $\ve > 0$
$$
\tau_\ve =
\begin{cases}
\tau + (\tl \sigma_\ve \circ \theta_\tau),& \text{if}\ \ X(\tau) \notin \partial \lo,\\
\tau + (\sigma_\ve \circ \theta_\tau),& \text{if}\ \ X(\tau) \in \partial \lo,\ F\big(s+\tau, X(\tau)\big) = H\big(s+\tau, X(\tau)\big),\\
\tau + (\sigma^c_\ve \circ \theta_\tau),& \text{if}\ \ X(\tau) \in \partial \lo,\ F\big(s+\tau, X(\tau)\big) = G\big(s+\tau, X(\tau)\big),
\end{cases}
$$
where $\tl \sigma_\ve = \inf\{u \ge 0:\ X(u) \notin \Gamma_\ve \}$ and $\sigma_\ve, \sigma^c_\ve$ are defined in assumption \ref{ass:sigma_ve}. The stopping time $\tau_\ve$ might attain the value $\infty$, in which case the functional $J$ is also well defined due to discounting. Notice the difference between $\tl \sigma_\ve$ and $\sigma_\ve$ ($\sigma^c_\ve$): the former is the first exit time from the $\ve$-neighbourhood $\Gamma_\ve$ of $\partial \lo$, whereas the latter is the first exit time from $\lo \cup \Gamma_\ve$ ($\lo^c \cup \Gamma_\ve$, resp.). The stopping time $\tau_\ve$ equals $\tau$ for appropriately small $\ve$ if $X(\tau) \notin \partial \lo$. Otherwise, i.e., when $X(\tau) \in \partial\lo$, it follows from assumption \ref{ass:sigma_ve} that $\lim_{\ve \to 0} \tau_\ve \to \tau$ $\prob^x$-a.s. If $F = H$ at the time $\tau$ then $X(\tau_\ve) \in \lo^c$ (if it is finite) and by the continuity of $H$ we obtain
$$
\lim_{\ve \to 0} e^{-\alpha \tau_\ve} H\big(\tau_\ve, X(\tau_\ve)\big) =  e^{-\alpha \tau} H\big(\tau, X(\tau)\big) \quad \text{$\prob^x$-a.s.}
$$
We proceed similarly when $F = G$ at the time $\tau$ and get
$$
\lim_{\ve \to 0}  e^{-\alpha \tau_\ve} F\big(\tau_\ve, X(\tau_\ve)\big) =  e^{-\alpha \tau} F\big(\tau, X(\tau)\big) \quad \text{$\prob^x$-a.s.}
$$
Dominated convergence theorem implies
$$
\lim_{\ve \to 0} J^\infty\big(s, x, \tau_\ve \big) = J^\infty(s, x, \tau).
$$
We also have $\prob^x\{ X(\tau_\ve) \in \partial \lo\} = 0$ for each $\ve > 0$. Hence,
\begin{align*}
J^\infty(s,x, \tau) &= \lim_{\ve \to 0} \Big[ J^\infty(s,x,\tau_\ve) - M \prob^x\{X(\tau_\ve) \in \partial \lo\} \Big]\le \sup_{\hat\tau} \Big[ J^\infty(s,x,\hat\tau) - M \prob^x\{X(\hat\tau) \in \partial \lo\} \Big].
\end{align*}
Combining this result with \eqref{eqn:M_and_boundary} yields
$$
J^\infty(s,x,\tau) \le w^{\infty, \infty}(s,x),
$$
which, due to arbitrariness of $\tau$, gives the required inequality $w^\infty \le w^{\infty, \infty}$.

\textbf{Step 3.} Using standard methods we extend above result to continuous bounded $G$ and $H$ in a similar way as in Theorem \ref{thm:cont_for_upjump}.

We relax now the assumption $F = G \vee H$ on $\partial \lo$. Let $F$ be as in the statement of the theorem and
$$
\tl F(s,x) =
\begin{cases}
F(s,x), & x \notin \partial \lo,\\
(G \vee H)(s,x), & x \in \partial \lo.
\end{cases}
$$
Denote by $\tl w^{\beta, \infty}$ and $\tl w^\infty$ the value functions corresponding to $\tl F$. Fubini's theorem and assumption \ref{ass:zero_at_boundary} imply that
$$
\ee^x \Big\{ \int_0^\infty e^{-\alpha u} \ind{X(u) \in \partial \lo} du \Big\} = 0.
$$
Using this equality we obtain
\begin{align*}
&w^{\beta, \infty}(s,x)\\
&= \ee^x \Big\{ \int_0^\infty e^{-\alpha u} \Big[f\big(s + u, X(u)\big) + \ind{X(u) \notin \partial \lo} \beta \big(F - w^{\beta,\infty}\big)^+\big(s + u, X(u)\big)\\
&\mop{90}+ \ind{X(u) \in \partial \lo} \beta \big(F - w^{\beta,\infty}\big)^+\big(s + u, X(u)\big) \Big]du \Big\}\\
&= \ee^x \Big\{ \int_0^\infty e^{-\alpha u} \Big[f\big(s + u, X(u)\big) + \ind{X(u) \notin \partial \lo} \beta \big(\tl F - w^{\beta,\infty}\big)^+\big(s + u, X(u)\big)\\
&\mop{90}+ \ind{X(u) \in \partial \lo} \beta \big(\tl F - w^{\beta,\infty}\big)^+\big(s + u, X(u)\big) \Big]du \Big\}\\
&= \ee^x \Big\{ \int_0^\infty e^{-\alpha u} \Big[f\big(s + u, X(u)\big) + \beta \big(\tl F - w^{\beta,\infty}\big)^+\big(s + u, X(u)\big)\Big]du \Big\},
\end{align*}
where the second equality follows from the fact that $F$ coincides with $\tl F$ on $E \setminus \partial \lo$ and the third term under the integral integrates to zero. Since $\tl w^{\beta, \infty}$ is a unique solution of the penalized equation \eqref{eqn:inf_penalized} with function $\tl F$ (see Lemma \ref{lem:inf_wbeta}) we conclude that $\tl w^{\beta, \infty} = w^{\beta, \infty}$.

Denote by $\tl J^\infty$ the functional $J^\infty$ with the function $\tl F$. Fix any stopping time $\tau$ and define for $\ve > 0$
$$
\tau_\ve =
\begin{cases}
\tau + \tl \sigma_\ve \circ \theta_\tau,& \text{if}\ \ X(\tau) \notin \partial \lo,\\
\tau + \sigma_\ve \circ \theta_\tau,& \text{if}\ \ X(\tau) \in \partial \lo,\ \tl F\big(s+\tau, X(\tau)\big) = H\big(s+\tau, X(\tau)\big),\\
\tau + \sigma^c_\ve \circ \theta_\tau,& \text{if}\ \ X(\tau) \in \partial \lo,\ \tl F\big(s+\tau, X(\tau)\big) = G\big(s+\tau, X(\tau)\big).
\end{cases}
$$
Similarly, as in Step 2 we obtain
$$
\lim_{\ve \to 0} J^\infty(s, x, \tau_\ve) = \tl J^\infty(s,x,\tau),
$$
which implies $\tl w^{\infty} \le w^\infty$. Opposite inequality is obvious as $\tl F \ge F$.

In the first part of the proof of assertion (4) $\tl w^{\beta, \infty}$ was shown to converge to $\tl w^\infty$ uniformly on compact sets in $[0, \infty) \times (E \setminus \partial \lo)$. Since $\tl w^{\beta, \infty}$ coincides with $w^{\beta, \infty}$ and $\tl w^\infty$ coincides with $w^\infty$ this uniform convergence holds for $w^{\beta, \infty}$ and $w^\infty$.
\end{proof}

\begin{remark}
The complexity of the proof of assertion (4) in Theorem \ref{thm:inf_cont_and_convergence} is caused by the incompatibility of the continuity conditions that one has to impose on the function $F$. On the one hand, we need to prove that $w^{\infty, \infty} \ge F$, which requires that $F$ is lower semicontinuous. On the other hand, the inequality $w^{\beta, \infty} \le w^\infty$ is true under the condition that $F$ is upper semicontinuous (see Lemma \ref{lem:wbetainfty_form_tau}).
\end{remark}

\section{Finite time horizon}\label{sec:finite_time_func}
Methods from previous section can be applied to optimal stopping of the following functional:
\begin{multline*}
J^T(s,x,\tau) = \ee^x \Big\{ \int_0^{\tau \wedge (T-s)}e^{-\alpha u} f\big(s + u, X(u)\big) du
 + e^{-\alpha (\tau \wedge (T-s))} F\big((s + \tau) \wedge T, X(\tau \wedge (T-s))\big) \Big\},
\end{multline*}
where $\alpha \ge 0$, the function $f$ is measurable bounded and continuous in $s$ uniformly in $x$ from compact sets, and $F$ has the following form: $F(s,x)=G(s,x)$ for $x\in \lo$ and $F(s,x)=H(s,x)$ for $x\in \lo^c \setminus \partial \lo$ for bounded continuous functions $G$ and $H$. This functional is a finite time horizon version of the functional $J^{\infty}(s,x,\tau)$.

Denote the value function by $w^T(s,x) = \sup_{\tau} J^T(s,x,\tau)$. Notice that $w^T$ can be equivalently written as
$$
w(s,T) = \sup_{\tau \le T-s} J^\infty(s,x,\tau).
$$

To enable numerical approximations of this value function we introduce a penalized equation:
\begin{multline}\label{eqn:finite_time_penalized_eqn}
w^{\beta, T}(s,x) = \ee^x \Big\{ \int_0^{T-s} e^{-\alpha u} \big[f + \beta \big(F - w^{\beta,T}\big)^+ \big]\big(s + u, X(u)\big)du
+ e^{-\alpha (T-s)} F\big(T, X(T-s)\big)\Big\}.
\end{multline}
\begin{lemma}\label{lem:finite_time}
Assume \ref{ass:strong_feller}.
\begin{enumerate}
\item There is a unique measurable bounded solution to \eqref{eqn:finite_time_penalized_eqn} and this solution is continuous on $[0, T) \times E$. Under Assumption \ref{ass:a2'} the continuity extends to $\big([0, T) \times E\big) \cup \big(T \times (E \setminus \partial\lo)\big)$.
\item This solution has an equivalent representation:
\begin{equation}\label{eqn:wbetaT_form_mbeta}
\begin{aligned}
&w^{\beta,T}(s,x)
= \sup_{b\in M_\beta} \ee^x\Big\{\int_0^{T-s}e^{-\alpha u - \int_0^{u}b(t)dt} \Big[f\big(s+u,X(u)\big)+ b(u)F\big(s+u,X(u)\big)\Big]du\\
&\mop{85} + e^{-\alpha (T-s)-\int_0^{T-s} b(t)dt} F\big(T, X(T-s)\big) \Big\},
\end{aligned}
\end{equation}
where $M_\beta$ is the class of progressively measurable processes with values in $[0,\beta]$.
\item $w^{\beta, T}$ is increasing in $\beta$.
\end{enumerate}
\end{lemma}
\begin{proof}
Similarly as in Lemma \ref{lem:solv_w_beta} we show that there is a unique bounded measurable solution to \eqref{eqn:finite_time_penalized_eqn}. Lemma \ref{lem:a3'} implies that this solution is continuous on $[0, T) \times E$. The continuity on $T \times (E \setminus \partial \lo)$ is more delicate. Fix $x \in \lo$ and a sequence $(s_n, x_n) \subset [0, T] \times \lo$ converging to $(T, x)$. We have
\begin{align*}
\big| w^{\beta, T} (s_n, x_n) - w^{\beta, T}(T, x) \big|
&\le
\big( \|f\| + \beta \|(F - w^{\beta, T})^+\|\big) (T-s_n) + 2 \|F\|\, \prob^{x_n} (\tau_\lo \le T-s_n)\\
&\mop{11}+
\big| \ee^{x_n} \big\{ e^{-\alpha (T-s_n)} G(T, X(T-s_n)) \big\} - G(T,x) \big|.
\end{align*}
The first term vanishes as $n \to \infty$. Assumption \ref{ass:a2'} implies that the second term converges to zero. The convergence of the third term follows from the continuity of the mapping (see Proposition \ref{cor:03}):
$$
(s,x,h) \mapsto \ee^x \{ e^{-\alpha h} G(s+h, X(h)) \}.
$$
Convergence to $x \in \lo^c \setminus \partial \lo$ can be proved in an analogous way.

Representation \eqref{eqn:wbetaT_form_mbeta} is obtained in an analogous way as in the proof of Lemma \ref{lem:wbeta_form_mbeta}. Assertion (3) follows immediately from \eqref{eqn:wbetaT_form_mbeta}.
\end{proof}

\begin{thm}\label{thm:main_finite_horizon}
Assume \ref{ass:a2'}, \ref{ass:strong_feller}.
\begin{enumerate}
\item The function $w^T$ is continuous on $[0, T] \times (E \setminus \partial \lo)$.
\item $w^{\infty,T}(s,x) := \lim_{\beta \to \infty}w^{\beta,T}(s,x)$ is lower semicontinuous on $[0, T) \times E$ with values in $\er \cup \{\infty\}$.
\item if $F$ is l.s.c., then $w^T$ is lower semicontinuous on $[0, T) \times E$ and $w^{\infty,T}\ge w^T$.
\end{enumerate}
Assume further \ref{ass:sigma_ve}, \ref{ass:zero_at_boundary}.
\begin{enumerate}
\setcounter{enumi}{3}
\item If $F \le G \vee H$ on $[0, T] \times \partial \lo$ then $w^{\beta, T}$ converges to $w^T$ as $\beta \to \infty$ uniformly on compact subsets of $[0, T] \times (E \setminus \partial \lo)$.
\item The mapping $(s, \infty) \ni T \mapsto w^T(s,x)$ is continuous for fixed $s$ and $x \in E \setminus \partial \lo$.
\end{enumerate}
\end{thm}
\begin{proof}
Similarly as in the proof of Theorem \ref{thm:inf_cont_and_convergence}, we show that $w^T$ is continuous on $[0, T) \times (E \setminus \partial \lo)$. The extension of the continuity to $T \times (E \setminus \partial \lo)$ follows an analogous route as in the proof of Lemma \ref{lem:finite_time}. Fix $x \in \lo$ and a sequence $(s_n, x_n) \subset [0, T] \times \lo$ converging to $(T, x)$. Then
\begin{align*}
\big| w^{T} (s_n, x_n) - w^{T}(T, x) \big|
&\le
\|f\| (T-s_n) + 2 \|F\|\, \prob^{x_n} (\tau_\lo \le T-s_n)\\
&\mop{11}+
\Big| \sup_{\tau \le (T-s_n)} \ee^{x_n} \big\{ e^{-\alpha \tau} G(s_n + \tau, X(\tau)) \big\} - G(T,x) \Big|.
\end{align*}
The first term vanishes as $n \to \infty$. Assumption \ref{ass:a2'} implies that the second term converges to zero. The convergence of the third term follows from the continuity of the mapping (see \cite[Corollary 3.6]{palczewski2008})
$$
(s,x,h) \mapsto \sup_{\tau \le h} \ee^x \{ e^{-\alpha \tau} G(s+\tau, X(\tau)) \}.
$$
Convergence to $x \in \lo^c \setminus \partial \lo$ can be proved in a similar way.

Assertions (2)-(3) are proved in a similar way as in Theorem \ref{thm:inf_cont_and_convergence}.

Assertion (4) is trivial for $s = T$ since $w^{\beta,T}(T, x) = w^{T}(T,x) = F(T,x)$. In the following we will address the case $s < T$. In a similar way as in the proof of Theorem \ref{thm:inf_cont_and_convergence} we show that it suffices to prove both assertions for $F = G \vee H$ on $\partial \lo$. Under this condition, as in Lemma \ref{lem:wbetainfty_form_tau}, we show that the function $w^{\beta,T}$ has the following equivalent representation:
\begin{multline}\label{eqn:wbetaT_form_tau}
w^{\beta, T} (s, x) = \sup_\tau \Big\{ J^T(s, x, \tau)
- \ee^x \big\{ e^{-\alpha (\tau \wedge (T-s))}
\big( F - w^{\beta, T} \big)^+\big((s+\tau) \wedge T, X(\tau \wedge (T-s))\big) \big\}\Big\}.
\end{multline}
This implies that $w^{\infty, T} := \lim_{\beta \to \infty} w^{\beta, T} \le w^T$. The proof of the opposite inequality requires similar but slightly more delicate argument as in the proof of Theorem \ref{thm:inf_cont_and_convergence}.

\textbf{Step 1.} Assume
\begin{align*}
&G(s,x) = \ee^x \Big\{ \int_0^{T-s} e^{-\alpha u} g\big(s+u, X(u)\big) du + e^{-\alpha (T-s)} G\big(T, X(T)\big) \Big\},\\
&H(s,x) = \ee^x \Big\{ \int_0^{T-s} e^{-\alpha u} h\big(s+u, X(u)\big) du + e^{-\alpha (T-s)} H\big(T, X(T)\big) \Big\},
\end{align*}
for continuous bounded functions $g, h$. Combining arguments from proofs of Lemmas \ref{lem:bar_w_beta} and \ref{lem:inf_difference_estimate} we obtain
\begin{align*}
&(G - w^{\beta, T})^+(s,x) \le \|f - g \| \phi^{\beta, T-s}_\lo (x) + \|F - G\| e^{-(\alpha + \beta)(T-s)},\\
&(H - w^{\beta, T})^+(s,x) \le \|f - h \| \phi^{\beta, T-s}_{\lo^c \setminus \partial \lo}(x) + \|F - H\| e^{-(\alpha + \beta)(T-s)},
\end{align*}
where, for an open set $A \in E$,
$$
\phi^{\beta, t}_A (x) = \ee^x \Big\{ \int_0^{t} e^{-\alpha u - \beta \int_0^u \ind{X(t) \in A} dt}du \Big\}.
$$
Above estimates and identity \eqref{eqn:wbetaT_form_tau} imply the following bound:
\begin{align*}
w^{\beta, T} (s,x) \ge \sup_{\tau \le T-s} \Big[ J^T(s,x,\tau)
&- \ee^x \Big\{ \ind{\tau = T-s} \big( F - w^{\beta, T} \big)^+ \big(T, X(\tau)\big) \Big\}\\
&- \ee^x \Big\{ \ind{\tau < T-s} \ind{X(\tau) \in \lo} \|f-g\| \phi^{\beta, T-(s + \tau)}_\lo\big(X(\tau)\big) \Big\}\\
&- \ee^x \Big\{ \ind{\tau < T-s} \ind{X(\tau) \in \lo^c \setminus \partial \lo} \|f-h\| \phi^{\beta, T-(s + \tau)}_{\lo^c \setminus \partial \lo} \big(X(\tau)\big) \Big\}\\
&- \ee^x \Big\{ \ind{\tau < T-s} \ind{X(\tau) \in \partial \lo} \frac{\|f-g\| \vee \|f - h\|}{\alpha} \Big\} \\
&- \ee^x \Big\{ \ind{\tau < T-s} \|G - H\| e^{-(\alpha + \beta)(T-(s + \tau))} \Big\}
\Big].
\end{align*}
Taking the limit as $\beta \to \infty$  and recalling that $w^{\beta, T}(T, x) = F(T,x)$ yield
\begin{equation}\label{eqn:M_and_boundary_T}
w^{\infty, T}(s,x) \ge \sup_{\tau \le T-s} \Big[ J^T(s,x,\tau) - M \prob^x\{X(\tau) \in \partial \lo\} \Big],
\end{equation}
where $M = (\|f - g\| \vee \|f - h\|)/\alpha$.

\textbf{Step 2.} Fix $\delta > 0$. For any $T > 0$, $s \in [0, T)$ and a stopping time $\tau \le T-s$ define for $\ve > 0$
$$
\tau_\ve =
\begin{cases}
\tau + (\tl \sigma_\ve \circ \theta_\tau) \wedge \delta,& \text{if}\ \ X(\tau) \notin \partial \lo,\\
\tau + (\sigma_\ve \circ \theta_\tau) \wedge \delta,& \text{if}\ \ X(\tau) \in \partial \lo,\ F\big(s+\tau, X(\tau)\big) = H\big(s+\tau, X(\tau)\big),\\
\tau + (\sigma^c_\ve \circ \theta_\tau) \wedge \delta,& \text{if}\ \ X(\tau) \in \partial \lo,\ F\big(s+\tau, X(\tau)\big) = G\big(s+\tau, X(\tau)\big),
\end{cases}
$$
where $\tl \sigma_\ve = \inf\{u \ge 0:\ X(u) \notin \Gamma_\ve \}$ and $\sigma_\ve, \sigma^c_\ve$ are defined in assumption \ref{ass:sigma_ve}. Contrary to the proof of Theorem \ref{thm:inf_cont_and_convergence} the difference $\tau_\ve - \tau$ is bounded by $\delta$. This, however, does not affect the limits as $\ve \to 0$. If $F = H$ at the time $\tau$ then $X(\tau_\ve) \in \lo^c$ for appropriately small $\ve$ and by continuity of $H$ we obtain
$$
\lim_{\ve \to 0} H\big(\tau_\ve, X(\tau_\ve)\big) = H\big(\tau, X(\tau)\big) \quad \text{$\prob^x$-a.s.}
$$
We proceed similarly when $F = G$ at the time $\tau$ and get
$$
\lim_{\ve \to 0} F\big(\tau_\ve, X(\tau_\ve)\big) = F\big(\tau, X(\tau)\big) \quad \text{$\prob^x$-a.s.}
$$
Dominated convergence theorem implies
$$
\lim_{\ve \to 0} J^\infty\big(s, x, \tau_\ve \big) = J^\infty(s, x, \tau), \qquad \text{and} \qquad \lim_{\ve \to 0} \prob^x\{ X(\tau_\ve) \in \partial \lo\} = 0.
$$
Recalling that $\tau_\ve \le T - s + \delta$ we obtain
$$
J^\infty(s,x,\tau_\ve) - M \prob^x\{ X(\tau_\ve) \in \partial \lo\} \le \sup_{\tau \le T - s + \delta}\big\{ J^\infty\big(s, x, \tau \big) - M \prob\{X(\tau) \in \partial \lo\} \big\}.
$$
As $\ve \to 0$ the left-hand side converges to $J^\infty(s,x,\tau)$. The arbitrariness of $\tau$ implies
\begin{equation}\label{eqn:JL_delta}
w^T(s, x) \le \sup_{\tau \le T - s + \delta}\big\{ J^\infty\big(s, x, \tau \big) - M \prob\{X(\tau) \in \partial \lo\} \big\}.
\end{equation}

\textbf{Step 3.} Combining formulas \eqref{eqn:M_and_boundary_T} and \eqref{eqn:JL_delta} we get for any $\delta > 0$
$$
w^{\infty, T}(s, x) \ge w^{T-\delta}(s, x), \qquad (s,x) \in [0, T-\delta] \times E.
$$
Take a stopping time $\tau \le T-s$ and define $\tau_\delta = \tau \wedge (T - s - \delta)$. We have
\begin{align*}
&\limsup_{\delta \to 0} \big| J^\infty(s,x,\tau_\delta) - J^\infty(s, x, \tau) \big|\\
&\le
\limsup_{\delta \to 0} \ee^x \big\{ \big| e^{-\alpha \tau_\delta} F(s,x,\tau_\delta) - e^{-\alpha \tau} F(s,x,\tau) \big| \big\}\\
&\le
\limsup_{\delta \to 0} \ee^x \big\{ \ind{\tau < T-s} \big| e^{-\alpha \tau_\delta} F(s,x,\tau_\delta) - e^{-\alpha \tau} F(s,x,\tau) \big| \big\}\\
&\mop{13}+
\limsup_{\delta \to 0} \ee^x \big\{ \ind{\tau = T-s,\ X(\tau) \in \partial \lo} \big| e^{-\alpha \tau_\delta} F(s,x,\tau_\delta) - e^{-\alpha \tau} F(s,x,\tau) \big| \big\}\\
&\mop{13}+
\limsup_{\delta \to 0} \ee^x \big\{ \ind{\tau = T-s,\ X(\tau) \notin \partial \lo} \big| e^{-\alpha \tau_\delta} F(s,x,\tau_\delta) - e^{-\alpha \tau} F(s,x,\tau) \big| \big\}\\
&= (1) + (2) + (3).
\end{align*}
Limit (1) equals $0$ from dominated convergence theorem. By quasi-left continuity of the process $X(t)$ and dominated convergence theorem limit (3) is $0$ as well. Term (2) is dominated by $2 \|F\| \prob^x\{X(T-s) \in \partial \lo\}$, which by \ref{ass:zero_at_boundary} is equal to $0$. Hence,
$$
\lim_{\delta \to 0} w^{T-\delta}(s,x) = w^T(s,x).
$$
This completes the proof of both assertions in the case when $G$ and $H$ can be written in resolvent forms (assertion (4) follows from Dini's theorem and assertion (1)).

\textbf{Step 4.} Using standard methods we extend above result to continuous bounded $G$ and $H$ in a similar way as in Theorem \ref{thm:cont_for_upjump}.

\textbf{Step 5.} We relax the assumption that $F = G \vee H$ on $[0, T] \times \partial \lo$ as in the proof of Theorem \ref{thm:inf_cont_and_convergence}.
\end{proof}

\begin{appendices}

\section{Properties of weak Feller processes} \label{sec:feller}

A Markov process defined on a locally compact separable space is called standard (see \cite{dynkin}, p. 104, or \cite[Definition 9.2]{blum}) if
\begin{enumerate}
\item it is a strong Markov process,
\item it is c\`adl\`ag and quasi-left-continuous,
\item the filtration is complete and right-continuous.
\end{enumerate}

Let $\big(X(t)\big)$ be a c\`adl\`ag Markov process defined on a locally compact separable space $(E, \bse)$ endowed with a metric $\rho$ with respect to which every closed ball is compact. Assume that this process satisfies the weak Feller property:
$$
P_t\, \mathcal{C}_0 \subseteq \mathcal{C}_0,
$$
where $\mathcal C_0$ is the space of continuous bounded functions $E \to \er$ vanishing in infinity, and $P_t h(x) = \ee^x \left\{h\big(X(t)\big)\right\}$ for any bounded measurable $h:E \to \er$. Right continuity of $\big(X(t)\big)$ and
Theorem T1, Chapter XIII in \cite{Meyer} implies that the semigroup $P_t$ satisfies the following uniform continuity
property:
\begin{equation*}
\label{eqn:unifcont}
\lim_{t \to 0+} P_t f  = f \ \ \text{in $\mathcal{C}_0$}, \qquad \forall\, f \in \mathcal{C}_0.
\end{equation*}
Theorem 3.1 (p. 104) in \cite{dynkin} implies that there exists a standard Markov process on the state space $E$ with the semigroup $P_t$. In fact, it follows from the proof of the aforementioned theorem that the process $(X(t))$ satisfies the conditions of a standard process if its filtration is complete. The filtration of $(X(t))$ can be completed without changes to other properties of the process due to Proposition \ref{cor:03} below and Theorem 3.3 and Subsection 3.6 in \cite{dynkin}.

Let
\begin{equation}
\gamma_T(x,R)=\prob^x\left\{\exists_{s\in [0,T]} \
\rho\big(x,X(s)\big)\geq R \right\}.
\end{equation}
\begin{prop} (\cite[Proposition 2.1]{palczewski2008})
\label{prop:01}
For any compact set $K \subseteq E$
\begin{equation}
\sup_{x\in K} \gamma_T(x,R)\to 0 \end{equation} as $R \to \infty$.
\end{prop}

\begin{prop} (\cite[Corollary 2.2]{palczewski2008})
\label{cor:03}$ $
\begin{itemize}
\item[i)] $P_t \mathcal{C} \subset \mathcal{C}$, where
$\mathcal{C}$ is the space of continuous bounded functions $E \to \er$ (the Feller property).
 \item[ii)] $\lim_{t \to 0} P_t f(x) = f(x)$ uniformly on compact subsets of $E$ for $f \in \mathcal{C}$.
\end{itemize}
\end{prop}

\begin{prop}
(\cite[Theorem 3.7]{dynkin})
\label{prop:fellercont}
For any compact set $K \subseteq E$ and any $\varepsilon, \delta > 0$ there is $h_0 > 0$ such that
$$
\sup_{0 \le h \le h_0} \ \sup_{x \in K}\, \prob^x \{ X(h) \notin B(x, \delta)\} < \varepsilon.
$$
\end{prop}

\end{appendices}


\begin{thebibliography}{8}
\bibitem{bassan2002a} Bassan B, Ceci C (2002) {Optimal stopping with discontinuous reward: regularity of the value function and viscosity solution}, Stochastics and Stochastics Reports 72, 55-77
\bibitem{bassan2002b} Bassan B, Ceci C (2002) {Regularity of the value function and viscosity solutions in optimal stopping problems for general Markov processes}, Stochastics and Stochastics Reports 74, 633-649
\bibitem{Bens-Lions} Bensoussan A, Lions JL (1978) {Applications des In\'{e}quations Variationnelles en contr\^{o}le Stochastique}, Dunod
\bibitem{bismut1977} Bismut JM, Skalli B (1977) {Temps d'arret optimal, th\'{e}orie g\'{e}n\'{e}rale des processus et processus de Markov}, Z. Wahrscheinlichkeitstheorie und Verw. Gebiete 39:4, 301-313
\bibitem{blum} Blumenthal RM, Getoor RK (1968) {Markov Processes and Potential Theory}, Academic Press
\bibitem{dynkin} Dynkin EB (1965) {Markov processes}, Vol. 1, Springer
\bibitem{flemingsoner} Fleming WH, Soner HM (1993) {Controlled Markov Processes and Viscosity Solutions}, Springer
\bibitem{elkarui1981} El Karoui N (1981) {Les aspects probabilistes du contr\^{o}le stochastique}, Ninth Saint Flour Probability Summer School -- 1979, Lecture Notes in Mathematics 876, Springer, 73-238
\bibitem{elkarui1982} El Karoui N, Lepeltier JP, Marchal B (1982) {Optimal stopping of controlled Markov processes}, in Advances in filtering and optimal stochastic control, Lecture Notes in Control and Information Sciences 42, Springer, 106-112
\bibitem{fakeev} Fakeev AG (1971) {On the question of the optimal stopping of a Markov process} (Russian), Teor. Verojatnost. i Primenen.  16, 708-710
\bibitem{KS} Karatzas I, Shreve S (1991) {Brownian motion and stochastic calculus}, Springer
\bibitem{KD} Kushner HJ, Dupuis P (1992) {Numerical methods for stochastic control problems in continuous time}, Springer-Verlag, New York
\bibitem{lamberton} Lamberton D. (2009) {Optimal stopping with irregular reward functions}, Stochastic Processes and their Applications 119, 3253-3284
\bibitem{mac} Mackevicius V (1973) {Passing to the limit in the optimal stopping problems of Markov processes}, Liet. Mat. Rink. 13:1, 115-128
\bibitem{menaldi} Menaldi JL (1980) {On the optimal stopping time problem for degenerate diffusions}, SIAM J. Control and Optimization 18:6, 697-721
\bibitem{mertens} Mertens JF (1973) {Strongly supermedian functions and optimal stopping},  Z. Wahrscheinlichkeitstheorie und Verw. Gebiete  26, 119--139.
\bibitem{Meyer} Meyer PA (1969) {Markov Processes}, Research and Training School, Indian Statistical Institute (translation by MN Sastry)
\bibitem{palczewski2008} Palczewski J, Stettner \L{} (2010) {Finite Horizon Optimal Stopping of Time-Discontinuous Functionals with Applications to Impulse Control with Delay}, SIAM J. Control and Optimization 48:8, 4874-4909
\bibitem{robin} Robin M (1978) {Controle impulsionnel des processus de Markov (Thesis)}, University of Paris IX
\bibitem{SZ} Stettner \L{}, Zabczyk J (1983) {Optimal Stopping for Feller Markov Processes}, Preprint No. 284, IMPAN, Warsaw
\bibitem{stettner} Stettner \L{} (2008) {Penalty method for finite horizon stopping problems}, to appear in SIAM J. Control and Optimization
\bibitem{Z} Zabczyk J (1984) {Stopping Problems in Stochastic Control}, Proc. ICM-83, Vol. II, PWN North Holland, 1425-1437
\end{thebibliography}
\end{document}